\documentclass{article}
\usepackage{cite} 
\usepackage{amsthm} 
\usepackage{amsfonts} 
\usepackage{amsmath} 
\usepackage{amscd} 
\usepackage{tikz} 
\usepackage{dutchcal} 
\usepackage{mathrsfs} 
\usepackage{tikz-cd} 
\usepackage{mathtools} 
\usepackage{amssymb} 
\usepackage{euscript}
\usepackage{mathtools}

\usepackage[nottoc,numbib]{tocbibind} 

\usepackage{natbib} 

\theoremstyle{definition} 

\newtheorem{thm}{Theorem}[section] 
\newtheorem{lemma}[thm]{Lemma}
\newtheorem{prop}[thm]{Proposition}
\newtheorem{cor}[thm]{Corollary}

\newtheorem{defn}[thm]{Definition}
\newtheorem{rmk}[thm]{Remark}
\newtheorem{note}[thm]{Notation}

\newcommand{\ds}{\displaystyle}

\makeatletter
\newcommand{\colim@}[2]{%
  \vtop{\m@th\ialign{##\cr
    \hfil$#1\operator@font colim$\hfil\cr
    \noalign{\nointerlineskip\kern1.5\ex@}#2\cr
    \noalign{\nointerlineskip\kern-\ex@}\cr}}%
}
\newcommand{\colim}{%
  \mathop{\mathpalette\colim@{\rightarrowfill@\textstyle}}\nmlimits@
}
\makeatother

\newcommand{\Coeq}{%
  \mathop{\operatorfont Coeq}}
\newcommand{\Eq}{%
  \mathop{\operatorfont Eq}}
\newcommand{\Hom}{
  \mathop{\operatorfont Hom}}

  {\begin{list}{}%
          {\setlength{\leftmargin}{#1}}%
          \item[]%
  }
  {\end{list}}

\DeclareMathOperator{\hoco}{hocolim}
\newcommand{\hocolim}[1]{{\ds\hoco}_{#1}}

\DeclareMathOperator{\srep}{srep}

\newcommand{\oldthm}[3]{\vspace{\topsep} 
\noindent \textbf{#1 \ref{#2}.} #3

\vspace{\topsep}}

\begin{document}
\title{The Canonical Grothendieck Topology and a Homotopical Analog}
\author{C. Lester}
\date{\today}
\maketitle

\begin{abstract}
We explore the canonical Grothendieck topology and a new homotopical analog.
First we discuss some background information, including defining a new 2-category called the \textit{Index-Functor Category} and a sieve generalization.
Then we discuss a specific description of the covers in the canonical topology and a homotopical analog.
Lastly, we explore the covers in the homotopical analog by obtaining some examples. 
\end{abstract}

\tableofcontents

\section{Introduction}

Let $\EuScript{M}$ be a simplicial model category.
We prove that there is a Grothendieck topology on $\EuScript{M}$ that captures information about certain kinds of homotopy colimits.
In the case of topological spaces, the covers in the Grothendieck topology include the open covers of the space and the set of simplicies mapping into the space.
There are times in the homotopy theory of topological spaces where these two covers can be used similary; this new Grothendieck topology provides an overarching structure where both these types of covers appear naturally.
\bigskip

Sieves will be of particular importance in this paper and so we start with a reminder of their definition and a reminder of the definition of a Grothedieck topology (in terms of sieves); both definitions follow the notation and terminology used by Mac Lane and Moerdijk in \cite{maclane}.

For any object $X$ of a category $\EuScript{C}$, we call $S$ a \textit{sieve on $X$} if $S$ is a collection of morphisms, all of whose codomains are $X$, that is closed under precomposition, i.e. if $f\in S$  and $f\circ g$ makes sense, then $f\circ g\in S$.
In particular, we can view a sieve $S$ on $X$ as a full subcategory of the overcategory $(\EuScript{C}\downarrow X)$.

A \textit{Grothendieck topology} is a function that assigns to each object $X$ a collection $J(X)$ of sieves such that
\begin{enumerate}
\item (Maximality) $\{ f \, | \, \text{codomain } f = X \} = (\EuScript{C}\downarrow X) \in J(X)$
\item (Stability) If $S\in J(X)$ and $f\colon Y\to X$ is a morphism in $\EuScript{C}$, then \\ $f^\ast S\coloneqq \{g\, | \, \text{codomain }g = Y,\, f\circ g\in S \} \in J(Y)$
\item (Transitivity) If $S\in J(X)$ and $R$ is any sieve on $X$ such that \\ $f^\ast R \in J(\text{domain } f)$ for all $f\in S$, then $R\in J(X)$.
\end{enumerate}

\bigskip

In SGA 4.2.2 Verdier introduced the canonical Grothendieck topology.
He defined the {\it canonical topology} on a category $\EuScript{C}$ to be the largest Grothendieck topology where all representable presheaves are sheaves. 
With such an implicit definition we naturally start to wonder how one can tell what collection of maps are or are not in the canonical topology.
In order to obtain a more explicit description of the canonical topology we define a notion of \textit{universal colim sieve}:

\oldthm{Definition}{def of cs and ucs}{For a category $\EuScript{C}$, an object $X$ of $\EuScript{C}$ and sieve $S$ on $X$, we call $S$ a {\it colim sieve} if $\colim_{S}{U}$ exists and the canonical map $\colim_{S}{U}\to X$ is an isomorphism.
(Alternatively, $S$ is a colim sieve if $X$ is the universal cocone under the diagram $U\colon S\to \EuScript{C}$.) 
Moreover, we call $S$ a {\it universal colim sieve} if for all arrows $\alpha\colon Y\to X$ in $\EuScript{C}$, $\alpha^\ast S$ is a colim sieve on $Y$.}

\noindent Then we prove that the collection of all univeral colim sieves forms a Grothendieck topology, which is precisely the canonical topology: 

\oldthm{Theorem}{ucs is a top}{Let $\EuScript{C}$ be any category. The collection of all universal colim sieves on $\EuScript{C}$ forms a Grothendieck topology.}

\vspace{-\topsep}
\oldthm{Theorem}{can top is ucs}{For any (locally small) category $\EuScript{C}$, the collection of all universal colim sieves on $\EuScript{C}$ is the canonical topology.}

\noindent Moreover, for `nice' catgories, we find a basis for the canonical topology:

\oldthm{Theorem}{hinting at basis for can top in special case}{Let $\EuScript{C}$ be a cocomplete category with pullbacks whose coproducts and pullbacks commute.
A sieve $S$ on $X$ is a (universal) colim sieve of $\EuScript{C}$ if and only if there exists some $\{A_\alpha\to X\}_{\alpha\in\EuScript{A}} \subset S$ where $\coprod_{\alpha\in\EuScript{A}} A_\alpha \to X$ is a (universal) effective epimorphism.}


Theorems \ref{ucs is a top} and \ref{can top is ucs} are folklore, and can be found in \cite{johnstone2002sketches}. We give new proofs using a technique that also works for the homotopical analog that is our main result.

Adapting the above notions to the homotopical setting, we are led to the following:

\oldthm{Definition}{def of hs and uhs}{For a model category $\EuScript{M}$, an object $X$ of $\EuScript{M}$ and sieve $S$ on $X$, we call $S$ a {\it hocolim sieve} if the canonical map $\hocolim{S}{U}\to X$ is a weak equivalence.
Moreover, we call $S$ a {\it universal hocolim sieve} if for all arrows $\alpha\colon Y\to X$ in $\EuScript{C}$, $\alpha^\ast S$ is a hocolim sieve.}

\vspace{-\topsep}
\oldthm{Theorem}{uhcs is a top}{For a simplicial model category $\EuScript{M}$, the collection of all universal hocolim sieves on $\EuScript{M}$ forms a Grothendieck topology, which we dub the \textit{homotopical canonical topology}.}

\noindent This homotopical analog of the canonical topology has one particular feature: it `contains' as examples both the open covers of a space and the set of simplicies mapping into the space, i.e.

\oldthm{Proposition}{open covers are uhcs}{For any topological space $X$ and open cover $\EuScript{U}$, the sieve generated by $\EuScript{U}$ is in the homotopical canonical topology.}

\vspace{-\topsep}
\oldthm{Corollary}{simplices are uhcs}{For any topological space $X$, the sieve generated by the set $\{ \Delta^n \to X\ |\ n\in \mathbb{Z}_{\geq 0}\}$ is in the homotopical canonical topology.}

\noindent 
There are times in the homotopy theory of topological spaces when the set of simplices mapping into a space and the open covers of a space act similarly; for example, we can compute cohomology with both (singular and \v{C}ech respectively, which are isomorphic when the space is `nice'), and both contexts support detection theorems for quasi-fibrations. 
The homotopical canonical topology provides an overarching structure where both these types of covers appear naturally.

\bigskip

\noindent\textit{Organization.}

We start by laying the groundwork:
In Section \ref{section background} we spend some time exploring preliminary results and definitions, which includes a discussion on effective epimorphisms.
In Section \ref{section IF cat stuff} we define a new 2-category of diagrams in $\EuScript{C}$; this will allow us to ``work with colimits'' without knowing which colimits exist.
Then we do some exploration of this category's $Hom$-sets and 2-morphisms.
Lastly, in Section \ref{section gen sieves} we define a generalization of a sieve, i.e.\ a special subcategory of the overcategory, and get a few results pertaining to this generalization.

We use these background results (2-categories, generalizations, etc.) in Section \ref{UCS and HCS form a GTop} to prove that the collection of universal colim sieves forms a Grothendieck topology.
Additionally in Section \ref{UCS and HCS form a GTop}, we prove that the collection of universal hocolim sieves forms a Grothendieck topology.
The similarities between these proofs are highlighted.

Lastly, in Sections \ref{section ucs and can top} and \ref{hcs in top} we explore some of the implications of Section \ref{UCS and HCS form a GTop}.
Specifically, in Section \ref{section ucs and can top} we prove that the canonical topology can be described using universal colim sieves and 
get a basis for the canonical topology on `nice' categories. 
And in Section \ref{hcs in top}, we find some examples of universal hocolim sieves on the category of topological spaces.
\bigskip

\noindent\textit{General Notation.}

\begin{note}\label{forgetful functor}
For any subcategory $S$ of $(\EuScript{C}\downarrow X)$, we will use $U$ to represent the forgetful functor $S\to\EuScript{C}$. For example, for a sieve $S$ on $X$, $U(f)=\text{domain}\ f$.
\end{note}

\begin{note}\label{note hom}
For any category $\EuScript{D}$ and any two objects $P,M$ of $\EuScript{D}$, we will write $\EuScript{D}(P,M)$ for $\text{Hom}_{\EuScript{D}}(P,M)$.
\end{note}

\begin{note}\label{generators}
We say that a sieve $S$ on $X$ is \textit{generated} by the morphisms $ \{f_\alpha\colon A_\alpha\to X\}_{\alpha\in\mathcal{A}}$ and write $S = \langle \{f_\alpha\colon A_\alpha\to X\}_{\alpha\in\mathcal{A}}\rangle$ if each $f\in S$ factors through one of the $f_\alpha$, i.e. if $f\in S$ then there exists an $\alpha\in\mathcal{A}$ and morphism $g$ such that $f = f_\alpha\circ g$.
\end{note}
\bigskip

\noindent\textit{Acknowledgements.}

This work is part of the author's doctoral dissertation at the University of Oregon.
The author is extremely grateful to their advisor, Dan Dugger, for all of his guidance, wisdom and patience.

\section{Preliminary Information}\label{section background}

This section contains the preliminaries for the rest of the document, starting with the following important definitions:

\begin{defn}\label{def of cs and ucs}
For a category $\EuScript{C}$, an object $X$ of $\EuScript{C}$ and sieve $S$ on $X$, we call $S$ a {\it colim sieve} if $\colim_{S}{U}$ exists and the canonical map $\colim_{S}{U}\to X$ is an isomorphism.
(Alternatively, $S$ is a colim sieve if $X$ is the universal cocone under the diagram $U\colon S\to \EuScript{C}$.) 
Moreover, we call $S$ a {\it universal colim sieve} if for all arrows $\alpha\colon Y\to X$ in $\EuScript{C}$, $\alpha^\ast S$ is a colim sieve on $Y$.
\end{defn}

\begin{rmk}
In \cite{johnstone2002sketches} Johnstone also defined sieves of this form but the term `effectively-epimorphic' was used instead of the term `colim sieve.'
\end{rmk}

\begin{defn}\label{def of hs and uhs}
For a model category $\EuScript{M}$, an object $X$ of $\EuScript{M}$ and sieve $S$ on $X$, we call $S$ a {\it hocolim sieve} if the canonical map $\hocolim{S}{U}\to X$ is a weak equivalence.
Moreover, we call $S$ a {\it universal hocolim sieve} if for all arrows $\alpha\colon Y\to X$ in $\EuScript{C}$, $\alpha^\ast S$ is a hocolim sieve.
\end{defn}


\subsection{Basic Results}\label{section ucs def and basic reference}



This section mentions some basic results, all of which
we believe 
are well-known folklore but we include them here for completeness.

\begin{lemma}\label{pb sieve gen set}
Suppose $\EuScript{C}$ is a category with all pullbacks. \\ Let $S = \langle \{g_\alpha\colon A_\alpha\to X \}_{\alpha\in\mathfrak{A}} \rangle$ be a sieve on object $X$ of $\EuScript{C}$ and $f\colon Y\to X$ be a morphism in $\EuScript{C}$. Then $f^\ast S = \langle \{A_\alpha\times_X Y \overset{\pi_2}{\longrightarrow} Y\}_{\alpha\in\mathfrak{A}} \rangle$.
\end{lemma}

\begin{proof} It is an easy exercise.
\end{proof}

\begin{prop}\label{colim is coeq}
Let $\EuScript{C}$ be a cocomplete category. 
For a sieve in $\EuScript{C}$ on $X$ of the form $S = \langle \{f_\alpha\colon A_\alpha\to X \}_{\alpha\in\mathfrak{A}} \rangle$ such that $A_i\times_X A_j$ exists for all $i,j\in\mathfrak{A}$,
$$\colim_{S}{U} \cong
\Coeq \left(
\begin{tikzcd}\displaystyle
\coprod_{(i,j)\in\mathfrak{A}\times\mathfrak{A}} A_i\times_X A_j \arrow[d, shift right=2] \arrow[d, shift left=2] \\
\displaystyle \coprod_{k\in\mathfrak{A}} A_k
\end{tikzcd}\right)
$$
where the left and right vertical maps are induced from the projection morphisms $\pi_1\colon A_i\times_X A_j \to A_i$ and $\pi_2\colon A_i\times_X A_j \to A_j$.
\end{prop}

\begin{proof}
Let $I$ be the category with objects $\alpha$ and $(\alpha,\beta)$ for all $\alpha,\beta\in\mathfrak{A}$ and unique non-identity morphisms $(\alpha,\beta)\to\alpha$ and $(\alpha,\beta)\to\beta$.
Define a functor $L\colon I\to S$ by $L(\alpha) = f_\alpha$ and $L(\alpha,\beta) = f_{\alpha,\beta}$ where $f_{\alpha,\beta}\colon A_\alpha\times_X A_\beta \to X$ is the composition $f_\alpha\circ \pi_1 = f_\beta\circ \pi_2$. 
It is an easy exercise to see that $L$ is final in the sense that for all $f\in S$ the undercategory $(f\downarrow L)$ is connected.
Thus by \cite[][Theorem 1, Section 3, Chapter IX]{cwm}
$$\colim_{S}{U}\cong \colim_{I}{UL}.$$
But by the universal property of colimits, $\colim_{I}{UL}$ is precisely the coequalizer mentioned above.

\end{proof}

\begin{lemma}\label{cs nice with isom}
Let $\EuScript{C}$ be a category. Then $S$ is a colim sieve on $X$ if and only if $f^\ast S$ is a colim sieve for any isomorphism $f\colon Y\to X$.
\end{lemma}

\begin{proof}
It is an easy exercise.
\end{proof}


Recall that a morphism $f\colon Y\to X$ is called an \textit{effective epimorphism} provided $Y\times_X Y$ exists, $f$ is an epimorphism and $c\colon \Coeq\left(Y\times_X Y\ \substack{\longrightarrow\\ \longrightarrow}\ Y\right)\to X$ is an isomorphism.
Note that this third condition actually implies the second because $f=c\circ g$ where $g\colon Y\to \Coeq\left(Y\times_X Y\ \substack{\longrightarrow\\ \longrightarrow}\ Y\right)$ is the canonical map. Indeed, $g$ is an epimorphism by an easy exercise and $c$ is an epimorphism since it is an isomorphism.

Additionally, $f\colon Y\to X$ is called a \textit{universal effective epimorphism} if $f$ is an effective epimorphism with the additional property that for every pullback diagram
\begin{center}
\begin{tikzcd}
W \arrow{r} \arrow{d}[left]{\pi_g} &
Y \arrow{d}[right]{f} \\
Z \arrow{r}[below]{g} &
X
\end{tikzcd}
\end{center}
$\pi_g$ is also an effective epimorphism.

\begin{rmk}
A morphism $f\colon A\to B$ is called a \textit{regular epimorphism} if it is a coequalizer of some pair of arrows.
When the pullback $A\times_B A$ of $f$ exists in the category $\EuScript{C}$, then it is easy to see that $f$ is a regular epimorphism if and only if $f$ is an effective epimorphism.
\end{rmk}

\begin{cor}\label{eff epi's gen colim sieves}
Let $\EuScript{C}$ be a cocomplete category with pullbacks.
If $$S = \langle \{f\colon Y\to X\}\rangle$$ is a sieve on $X$, then $S$ is a colim sieve if and only if $f$ is an effective epimorphism. Moreover, $S$ is a universal colim sieve if and only if $f$ is a universal effective epimorphism.
\end{cor}

\begin{proof}
The condition for $f$ to be an effective epimorphism is, by Proposition \ref{colim is coeq}, precisely what it means for $S$ to be a colim sieve.
\end{proof}

\subsection{Effective Epimorphisms}\label{section eff epi}

Now we take a detour away from (universal) colim sieves to discuss some results about effective epimorphisms, which will be used in the proof of Theorem \ref{basis for can top in special case}. We start with a terminology reminder \cite[see][]{kashiwara2006categories}:
we call $f\colon A \to B$ a \textit{strict epimorphism} if any morphism $g\colon A \to C$ with the property that $gx = gy$ whenever $fx = fy$ for all $D$ and $x,y\colon D\to A$, factors uniquely through $f$, i.e. $g = hf$ for some unique $h\colon B \to C$.

\begin{prop}\label{un eff epis are un strict epis}
If the category $\EuScript{C}$ has all pullbacks, then a morphism $f$ is an effective epimorphism if and only if $f$ is a strict epimorphism.
\end{prop}

\begin{proof}
Let $f\colon A \to B$ be our morphism.
First suppose that $f$ is an effective epimorphism.
Let $g\colon A \to C$ be a morphism with the property that $gx = gy$ whenever $fx = fy$. 
Since $f$ is an effective epimorphism, then the commutative diagram
\begin{center}
\begin{tikzcd}
A \times_B A
\arrow{r}[above]{\pi_1}
\arrow{d}[left]{\pi_2} &
A
\arrow{d}[right]{f} \\
A
\arrow{r}[below]{f} &
B
\end{tikzcd}
\end{center}
%
is both a pushout and pullback diagram.
Since the diagram is commutative, i.e. $f\pi_1 = f\pi_2$, then $g\pi_1 = g\pi_2$.
Now the universal property of pushouts implies that there exists a unique $h\colon B \to C$ such that $g = hf$.
Hence $f$ is a strict epimorphism. 

To prove the converse, suppose that $f$ is a strict epimorphism. Consider the diagram
\begin{center}
$\EuScript{F} \coloneqq \left\{
\begin{tikzcd}
A \times_B A
\arrow[r, yshift = 1ex, "\pi_1"]
\arrow[r, yshift = -1ex, "\pi_2"'] &
A
\end{tikzcd}\right\}.$
\end{center}
%
We will show that $B$ is $\Coeq(\EuScript{F})$ by showing that $B$ satisfies the univeral property of colimits with respect to $\EuScript{F}$.
Specifically, suppose we have a morphism $\EuScript{F} \to C$, i.e. there is a morphism $g\colon A \to C$ such that $g\pi_1 = g\pi_2$.

Suppose we know $gx = gy$ whenever $x,y\colon D\to A$ and $fx = fy$.
Then, since $f$ is strict, this implies that there exists a unique $h\colon B \to C$ such that $g = hf$.
Hence, $B$ satisfies the universal property of colimits and so $B \cong \Coeq{\EuScript{F}}$.

Thus to show that $f$ is an effective epimorphism, it suffices to show:
\begin{center} if $x,y\colon D\to A$ and $fx = fy$, then $gx = gy$.\end{center}
For a fixed pair $x,y\colon D\to A$ such that $fx = fy$, we have the commutative diagram
\begin{center}
\begin{tikzcd}
D
\arrow{r}[above]{x}
\arrow{d}[left]{y} &
A \arrow{d}[right]{f} \\
A \arrow{r}[below]{f} &
B
\end{tikzcd}
\end{center}
%
Thus, by the universal property of pullbacks, both $x$ and $y$ factor through the pullback $A \times_B A$, i.e. $x = \pi_1 \alpha$ and $y = \pi_2\alpha$ for some unique morphism $\alpha\colon D \to A \times_B A$.
Therefore, our assumption $g\pi_1 = g\pi_2$ implies
$$gx = g\pi_1\alpha = g\pi_2\alpha = gy.$$
Hence $g$ has the property that $gx = gy$ whenever $fx = fy$.
\end{proof}

\begin{cor}\label{eff epi composition}
If the category $\EuScript{C}$ has all pullbacks, then universal effective epimorphisms are closed under composition.
\end{cor}

\begin{proof}
In \cite[][Proposition 5.11]{kellymono} Kelly proves that totally regular epimorphisms are closed under composition;
our Corollary follows immediately from Kelly's result and our Proposition \ref{un eff epis are un strict epis}.
We will end with a few remarks: what Kelly called regular epimorphisms are what we are calling strict epimorphisms, and
Kelly's \textit{totally} condition is precisely our \textit{universal} condition.
\end{proof}

Before our next result, we review some definitions.
Let $\EuScript{E}$ be a category with small hom-sets, all finite limits and all small colimits.
Let $E_\alpha$ be a family of objects in $\EuScript{E}$ and $E = \amalg_\alpha E_\alpha$.

The coproduct $E$ is called \textit{disjoint} if every coproduct inclusion $i_\alpha\colon E_\alpha\to E$ is a monomorphism and, whenever $\alpha\neq\beta$, $E_\alpha\times_E E_\beta$ is the initial object in $\EuScript{E}$.

The coproduct $E$ is called \textit{stable} (under pullback) if for every 
$f\colon D\to E$ in $\EuScript{E}$, the morphisms $j_\alpha$ obtained from the pullback diagrams
\begin{center}
\begin{tikzcd}
D\times_E E_\alpha \arrow{r} \arrow{d}[left]{j_\alpha} &
E_\alpha \arrow{d}[right]{i_\alpha} \\
D \arrow{r}[below]{f} &
E
\end{tikzcd}
\end{center}
%
induce an isomorphism $\coprod_\alpha(D\times_E E_\alpha)\cong D$.

\begin{rmk}\label{stable implies coproducts and pullbacks commute}
If every coproduct in $\EuScript{E}$ is stable, then the pullback operation $-\times_E D$ ``commutes'' with coproducts, i.e.\ 
$(\coprod_\alpha B_\alpha)\times_E D \cong \coprod_\alpha (B_\alpha\times_E D)$.
\end{rmk}

\begin{rmk}\label{initial object condition}
If a category $\EuScript{C}$ with an initial object $\emptyset$ has stable coproducts, then the existance of an arrow $X\to\emptyset$ implies $X\cong\emptyset$.
Indeed, consider $\EuScript{C}(X,Z)$, which has at least one element since it contains the composition $X\to\emptyset\to Z$.
We will prove that any two elements $f,g\in\EuScript{C}(X,Z)$ are equal.

By Remark \ref{stable implies coproducts and pullbacks commute},
$X\cong  X \times_\emptyset \emptyset \cong X\times_\emptyset (\emptyset\amalg\emptyset) \cong (X\times_\emptyset \emptyset)\amalg (X\times_\emptyset \emptyset) \cong X\amalg X $.
Let $\phi$ represent this isomorphism $X \amalg X\to X$.
Let $i_0$ and $i_1$ be the two natural maps $X\to X\amalg X$.
Then $id_X = \phi i_0$ and $id_X = \phi i_1$.
But $\phi$ is an isomorphism and so $i_0 = i_1$.

Now use $f$ and $g$ to induce the arrow $f\amalg g\colon X \amalg X\to Z$, i.e. $(f\amalg g)i_0 = f$ and $(f\amalg g)i_1 = g$.
Since $i_0 = i_1$, then $f = g$.


\end{rmk}

\begin{lemma}\label{coprod of eff epi's}
Let $\EuScript{C}$ be a category with disjoint and stable coproducts, and an initial object. 
Suppose $f_\alpha\colon A_\alpha\to B_\alpha$ are effective epimorphisms for all $\alpha\in\EuScript{A}$. Then $\coprod_{\EuScript{A}} f_\alpha \colon \coprod_{\EuScript{A}} A_\alpha\to \coprod_{\EuScript{A}} B_\alpha$ is an effective epimorphism (provided all necessary coproducts exist). Moreover, if $\EuScript{C}$ has all pullbacks and coproducts, and the $f_\alpha$ are universal effective epimorphisms, then $\coprod_{\EuScript{A}} f_\alpha$ is also a universal effective epimorphism.
\end{lemma}

\begin{proof}
Our basic argument is
\begin{align*}
\coprod_{\alpha\in\EuScript{A}} B_\alpha
& \cong \coprod_{\alpha\in\EuScript{A}} \Coeq
\left(\begin{tikzcd}
A_\alpha \times_{B_\alpha} A_\alpha
\arrow[d, shift right = 2]
\arrow[d, shift left = 2] \\
A_\alpha
\end{tikzcd}\right)\\
& \cong \Coeq
\left(\begin{tikzcd}
\coprod_{\alpha\in\EuScript{A}} \left(A_\alpha \times_{B_\alpha} A_\alpha\right)
\arrow[d, shift right = 2]
\arrow[d, shift left = 2] \\
\coprod_{\alpha\in\EuScript{A}} A_\alpha
\end{tikzcd}\right)\\
& \cong \Coeq
\left(\begin{tikzcd}
\left(\coprod_{\alpha\in\EuScript{A}} A_\alpha\right) \times_{\coprod_{\beta\in\EuScript{A}}B_\beta} \left(\coprod_{\gamma\in\EuScript{A}} A_\gamma\right)
\arrow[d, shift right = 2]
\arrow[d, shift left = 2] \\
\coprod_{\eta\in\EuScript{A}} A_\eta
\end{tikzcd}\right)
\end{align*}
The first isomorphism comes from assuming the $f_\alpha$ are effective epimorphisms.
The second isomorphism comes from commuting colimits.
The last isomorphism comes from the isomorphism
\begin{equation}\label{eff epi isom}
\coprod_{\alpha\in\EuScript{A}} \left(A_\alpha\times_{B_\alpha}A_\alpha\right) \cong \left(\coprod_{\alpha\in\EuScript{A}} A_\alpha\right) \times_{\coprod_{\beta\in\EuScript{A}}B_\beta} \left(\coprod_{\gamma\in\EuScript{A}} A_\gamma\right)
\end{equation}
which we will now justify.

Let $B = \coprod_{\beta\in\EuScript{A}} B_\beta$.
Since we know $\coprod_{\alpha\in\EuScript{A}} \left(A_\alpha\times_{B_\alpha}A_\alpha\right)$ exists, we will start here.
First we will show that $A_\alpha \times_{B_\alpha} A_\alpha\cong A_\alpha\times_B A_\alpha$ by showing that the object $A_\alpha \times_{B_\alpha} A_\alpha$, which we know exists, satisfies the requirements of $\lim (A_\alpha\,\substack{\to\\\to}\,B)$, which we have not assumed exists.
Notice that our maps $A_\alpha\xrightarrow{\sigma_\alpha} B$ factor as $A_\alpha\overset{f_\alpha}{\rightarrow} B_\alpha \overset{i_\alpha}{\rightarrow} B$, where the $i_\alpha$'s are the canonical inclusion maps.
This implies $A_\alpha \times_{B_\alpha} A_\alpha$ maps to the diagram $(A_\alpha\,\substack{\to\\\to}\,B)$ appropriately.
Now consider the parallel arrows $g,h\colon D\to A_\alpha$ such that $\sigma_\alpha g = \sigma_\alpha h$.
By the factorization, $i_\alpha f_\alpha g = i_\alpha f_\alpha h$.
Since $i_\alpha$ is a monomorphism, then $f_\alpha g = f_\alpha h$.
Now the universal property of the pullback $A_\alpha\times_{B_\alpha} A_\alpha$ gives us a unique map $D\to A_\alpha\times_{B_\alpha} A_\alpha$ that factors both $g$ and $h$ as desired.
Hence $A_\alpha \times_{B_\alpha} A_\alpha$ is $\lim (A_\alpha\,\substack{\to\\\to}\,B)$. 
Therefore $\coprod_{\alpha\in\EuScript{A}} \left(A_\alpha\times_{B_\alpha}A_\alpha\right) \cong \coprod_{\alpha\in\EuScript{A}} \left(A_\alpha\times_{B}A_\alpha\right).$

Since coproducts are disjoint, then $B_\alpha \times_B B_\gamma = \emptyset$ whever $\alpha\neq\gamma$.
Thus by Remark \ref{initial object condition} and the following diagram
\begin{center}
\begin{tikzcd}
A_\alpha \times_B A_\gamma \arrow{rr} \arrow[dr, dotted, "\exists"] \arrow{dd}
& & A_\gamma \arrow{d}[right]{f_\gamma} \\
& B_\alpha\times_B B_\gamma \arrow{d} \arrow{r} & B_\gamma \arrow{d}[right]{i_\gamma} \\
A_\alpha \arrow{r}[below]{f_\alpha} & B_\alpha \arrow{r}[below]{i_\alpha} & B
\end{tikzcd}
\end{center}
we see that $A_\alpha\times_B A_\gamma = \emptyset$ whenever $\alpha\neq\gamma$.
This implies that $$\coprod_{\alpha\in\EuScript{A}} \left(A_\alpha\times_{B}A_\alpha\right) \cong \coprod_{\alpha,\gamma\in\EuScript{A}} \left(A_\alpha\times_B A_\gamma\right).$$

Lastly, the commutativity of coproducts and pullbacks (see Remark \ref{stable implies coproducts and pullbacks commute}) yields $$\coprod_{\alpha,\gamma\in\EuScript{A}} \left(A_\alpha\times_B A_\gamma\right) \cong \coprod_{\alpha\in\EuScript{A}} A_\alpha \times_{B} \coprod_{\gamma\in\EuScript{A}} A_\gamma$$
which completes the justification of (\ref{eff epi isom}).

We have now shown that $\coprod_{\EuScript{A}} f_\alpha$ is an effective epimorphism.
The universality of $\coprod_{\EuScript{A}} f_\alpha$ is a consequence of the disjoint and stable coproducts.
Indeed, suppose $\EuScript{C}$ has all pullbacks and let $D\to B$ be a given morphism. Stability of coproducts implies that $D\cong \amalg_{\alpha\in\EuScript{A}} (D\times_B B_\alpha)$. It follows that the following is a pullback square
\begin{center}
\begin{tikzcd}
\coprod_{\alpha\in\EuScript{A}} (D\times_B B_\alpha \times_{B_\alpha} A_\alpha) \arrow{d}[left]{g} \arrow{rr} & &
\coprod_{\alpha\in\EuScript{A}} A_\alpha \arrow{d}[right]{\coprod f_\alpha} \\
\coprod_{\alpha\in\EuScript{A}} D\times_B B_\alpha \arrow[equal]{r} & D \arrow{r} &
\coprod_{\alpha\in\EuScript{A}} B_\alpha
\end{tikzcd}
\end{center}
where $g = \coprod_{\alpha\in\EuScript{A}} g_\alpha$ and $g_\alpha\colon D\times_B B_\alpha\times_{B_\alpha} A_\alpha \to D \times_B B_\alpha$ is the natural map.
Moreover, $g_\alpha$ is the pullback of the universal effective epimorphism $f_\alpha$.
Thus each $g_\alpha$ is an effective epimorphism and so we have already shown that $\amalg_{\alpha} g_\alpha = g$ is a an effective epimorphism.
\end{proof}



\section{Index-Functor Category}\label{section IF cat stuff}

In this section we will reframe what it means to be a `diagram in $\EuScript{C}$' by defining and discussing a special 2-category.
This 2-category will serve as a key tool in our manipulation of colimits and in proving that certain collections form Grothendieck topologies.  
\bigskip

For a fixed category $\EuScript{C}$, define $\mathscr{A}_\EuScript{C}$ to be the following 2-category:
\begin{itemize}
\item An object is a pair $(I,F)$ where $I$ is a small category and $F\colon I\to\EuScript{C}$ is a functor.
\item A morphism is a pair $(g,\eta)\colon (I,F) \to (I',F')$.
The $g$ is a functor $g\colon I\to I'$.
The $\eta$ is a natural transformation $\eta\colon F \to F'\circ g$.
Morally, we think of $g$ as almost being an arrow in $(Cat\downarrow\EuScript{C})$ where $Cat$ is the category of small categories; the natural transformation $\eta$ replaces the commutativity required for an arrow in the overcategory.
\item A 2-morphism from $(f,\eta_f)\colon  (I,D) \to (J,E)$ to $(g,\eta_g)\colon  (I,D) \to (J,E)$ is a natural transformation $\theta\colon f \to g$ such that for each $i$ in the objects of $I$, the following is a commutative diagram
\begin{center}
\begin{tikzcd}
& Di \arrow{ld}[above left]{(\eta_f)_i} \arrow{rd}[above right]{(\eta_g)_i} & \\
Efi \arrow{rr}[below]{E\theta i} & & Egi .
\end{tikzcd}
\end{center}
\end{itemize}

\begin{defn}
We call $\mathscr{A}_\EuScript{C}$ the {\it Index-Functor Category} for $\EuScript{C}$.
\end{defn}

\begin{note}\label{trivial cat note}
Let $\ast$ be the category consisting of one object and no non-identity morphisms. We will abuse notation and also use $\ast$ to represent its unique object. 
\end{note}
\begin{note}
For any object $Z$ of $\EuScript{C}$, let $cZ$ be the object of $\mathscr{A}_\EuScript{C}$ given by $(\ast,c_Z)$ where $c_Z(\ast) = Z$, i.e. $cZ$ is the constant diagram on $Z$.
\end{note}
\begin{note}
For a sieve $T$ on $X$, we will use $T$ as shorthand notation for the object $(T,U)$ of $\mathscr{A}_\EuScript{C}$.
(See Notation \ref{forgetful functor} for the definition of $U$.)
\end{note}
\begin{note}\label{cocone sieve map}
Let $T$ be a sieve on $X$.
We have a canonical map $\phi_T\colon T \to cX$ given by $\phi_T = (t,\varphi_T)$ where $t$ is the terminal map $T\to\ast$ and $\varphi_T\colon U \to (c_X \circ t)$ is given by $(\varphi_T)_f = f$ for $f\in T$.
\end{note}

\begin{rmk}\label{maps in IF on constants}
Notice that for all objects $V$ and $W$ of $\EuScript{C}$,
$$\mathscr{A}_\EuScript{C}(cV,cW) \cong \EuScript{C}(V,W)$$
since the only non-determined information in a map from $cV$ to $cW$ is the natural transformation $c_V\to c_W\circ t$, which is just a map $V\to W$ in $\EuScript{C}$.
In particular, we can view the $Hom$-sets in $\mathscr{A}_\EuScript{C}$ as a generalization of the $Hom$-sets in $\EuScript{C}$.
\end{rmk}

\subsection{$Hom$-sets}

The $Hom$-sets in $\mathscr{A}_\EuScript{C}$ will be particularly useful in our manipulation of colimits (as the following Lemma showcases). We use this section to discuss some of their properties.

\begin{lemma}\label{reword colim defn}
If $D\colon I\to\EuScript{C}$ and $X$ is a cocone for $D$, then we have an induced morphism $\phi\colon (I,D)\to cX$ in $\mathscr{A}_\EuScript{C}$.
The object $X$ is a colimit for $D$ if and only if the induced morphism $\phi^\ast\colon\mathscr{A}_\EuScript{C}(cX,cY)\to\mathscr{A}_\EuScript{C}((I,D),cY)$ is a bijection for all objects $Y$ of $\EuScript{C}$.
\end{lemma}

\begin{proof}
Left to the reader.
\end{proof}

\begin{lemma}\label{ll triangle commutes}
Let $(f,\eta_f),(g,\eta_g)\colon (I,F)\to (J,G)$ be two morphisms in $\mathscr{A}_\EuScript{C}$.
If there exists a 2-morphism $\alpha\colon (f,\eta_f)\to (g,\eta_g)$, then the induced maps \\
$(f,\eta_f)^\ast, (g,\eta_g)^\ast\colon \mathscr{A}_\EuScript{C}((J,G),cY) \to \mathscr{A}_\EuScript{C}((I,F),cY)$ are equal for all objects $Y$ in $\EuScript{C}$.
\end{lemma}

\begin{proof}
Let $(k,\eta_k)\in \mathscr{A}_\EuScript{C}((J,G),cY)$.
Then
$$(f,\eta_f)^\ast (k,\eta_k) = (k\circ f, f^\ast(\eta_k)\circ\eta_f) \quad
\text{and} \quad
(g,\eta_g)^\ast (k,\eta_k) = (k\circ g, g^\ast(\eta_k)\circ\eta_g).$$
%
But $k$ must be the terminal functor $J\to\ast$ and thus $k\circ f = k\circ g$.
To see that $ f^\ast(\eta_k)\circ\eta_f = g^\ast(\eta_k)\circ\eta_g$ fix an object $i\in I$ and notice that we have the following diagram:
\begin{center}
\begin{tikzcd}
& F_i \arrow{dl}[above left]{(\eta_f)_i} \arrow{dr}[above right]{(\eta_g)_i} & \\
Gfi \arrow{rr}[below]{G\alpha_i} \arrow{dr}[below left]{\eta_k} & &  
Ggi \arrow{dl}[below right]{\eta_k} \\ 
& Y &
\end{tikzcd}
\end{center}
where the upper part of the diagram commutes because $\alpha$ is a 2-morphism and the lower part commutes because of the natural transformation $\eta_k$.
Since the left vertical composition in the above diagram is $ (f^\ast(\eta_k)\circ\eta_f)_i$ and the right vertical composition is $ (g^\ast(\eta_k)\circ\eta_g)_i$, then this completes the proof.
\end{proof}


Before the last result we include a reminder about Grothendieck constructions.
Whenever we have a functor $G\colon A\to Cat$, where $Cat$ is the category of small categories, we can create a \textit{Grothendieck construction} of $G$, which we will denote $\text{Gr}(G)$.
The objects of $\text{Gr}(G)$ are pairs $(a,\tau)$ where $a$ is an object of $A$ and $\tau$ is an object of $G(a)$.
The morphisms are pairs $(f,g)\colon (a,\tau)\to(a',\tau')$ where $f\colon a\to a'$ is a morphism in $A$ and $g\colon Gf(\tau) \to\tau'$ is a morphism in $G(a')$.

\begin{prop}\label{technical lemma}
Let $A$ and $\mathcal{C}$ be categories.
Suppose there exists functors $G\colon A\to Cat$, $\theta\colon A\to \mathcal{C}$ and $\sigma\colon \text{Gr}(G)\to \mathcal{C}$, and a morphism in $\mathscr{A}_\EuScript{C}$ of the form $F=(f,\eta)\colon (\text{Gr}(G), \sigma)\to(A,\theta)$ where $f(a,\tau) = a$.
If for all objects $a$ of $A$, 
$\theta(a)$  is the colimit of $\sigma(a,-)\colon G(a)\to \EuScript{C}$ where the isomorphism is induced by $\eta$,  
then the induced map $F^\ast\colon \mathscr{A}_\EuScript{C}((A,\theta),cY)\to \mathscr{A}_\EuScript{C}((\text{Gr}(G),\sigma),cY)$ is a bijection for all objects $Y$ of $\mathcal{C}$.
\end{prop}

Remark: Fix $a\in A$, then $\eta_{(a,-)}\colon \sigma(a,-)\to\theta(a)$ is a natural transformation. In particular, $(\theta(a),\eta_{(a,-)})$ is a cocone under $\sigma(a,-)$. Our colimit assumption is specifically that this cocone is universal.

\begin{proof}
We start by showing that $F^\ast$ is an injection;
let $(k,\chi_k),(l,\chi_l)$ be in $\mathscr{A}_\EuScript{C} ((A,\theta), cY)$ such that $ F^\ast(k,\chi_k) =
F^\ast (l,\chi_l)$. In other words, suppose that
$(k\circ f, f^\ast(\chi_k)\circ\eta) =
(l\circ f, f^\ast(\chi_l)\circ\eta)$.
Since both $k$ and $l$ are functors $A\to \ast$ then they are both the terminal map, which is unique and hence $k=l$.

Now fix $a\in A$. Consider $(a,\tau)\in\text{Gr}(G)$.
For both $t=k$ and $t=l$, the natural transformations (i.e. second coordinates of the maps in question) at $(a,\tau)$ take the form
$$(f^\ast(\chi_t)\circ\eta)_{(a,\tau)} =
(\chi_t)_{a}\circ\eta_{(a,\tau)}\colon \sigma(a,\tau)\rightarrow \theta f(a,\tau) = \theta(a) \rightarrow c_Y t(a) = Y$$
where $c_Y$ comes from $cY = (\ast,c_Y)$.
Moreover, since $\eta$ and $\chi_t$ are both natural transformations, then these maps $\sigma(a,\tau)\to Y$ are compatible among all arrows in $G(a)$.
But by assumption $\colim_{G(a)}\sigma(a,-)\cong \theta(a)$.
Thus the maps $(\chi_t)_{a}\circ\eta_{(a,\tau)}$ define a map from the colimit, i.e. from $\theta(a)$ to $Y$.
By the universal property of colimits, there is only one choice for this map, namely $(\chi_t)_a$.
Moreover, since $(\chi_k)_{a}\circ\eta_{(a,\tau)} = (\chi_l)_{a}\circ\eta_{(a,\tau)}$, then $(\chi_k)_{a}$ and $(\chi_l)_{a}$ must define the same map out of the colimit.
Therefore $(\chi_k)_{a} = (\chi_l)_{a}$ for all $a\in A$ and this finishes the proof of injectivity.

To prove surjectivity, let $(m,\chi_m)\in\mathscr{A}_\EuScript{C} ((\text{Gr}(G),\sigma), cY)$. Let $(k,\chi_k)$ be the following pair:
\begin{itemize}
\item $k\colon A\to\ast$ is the terminal functor
\item $\chi_k$ is a collection of maps, one for each object $a$ of $A$, from $\theta(a)$ to $Y$.
The map for object $a$ is induced by the maps $(\chi_m)_{(a,\tau)}\colon  \sigma(a,\tau) \to Y$ for all $\tau$ in $G(a)$.
Note that these maps exist and are well defined because $\chi_m$ is a natural transformation and $\colim_{G(a)}{\sigma(a,-)} \cong \theta(a)$.
\end{itemize}
We claim two things: $(k,\chi_k)\in \mathscr{A}_\EuScript{C} ((A,\theta), cY)$ and $F^\ast (k,\chi_k) = (m,\chi_m)$

To prove the first claim we merely need to show that $\chi_k$ is a natural transformation
$\theta\to c_Y\circ k$.
By its definition, it is clear that $\chi_k$ does the correct thing on objects;
all we need to check is what it does to arrows in $A$.
Specifically, let $g\colon a\to b$ be a morphism in $A$.
Then for any $\tau\in G(a)$, $(g,id_{Gg(\tau)})$ is a morphism in $\text{Gr}(G)$.
Since $\chi_m\colon \sigma\to c_Y \circ m$ is a natural transformation, then we have the following commutative diagram
\begin{center}
\begin{tikzcd}
\sigma(a,\tau) \arrow{r}[above]{\sigma(g,id)} \arrow{d}[left]{\chi_m} & \sigma(b,Gg(\tau)) \arrow{d}[right]{\chi_m}\\
Y \arrow{r}[below]{id} & Y
\end{tikzcd}
\end{center}
and in particular, the map from diagram $\sigma(a,-)\colon G(a)\to\mathcal{C}$ to $Y$ factors through the map from diagram $\sigma(b,-)\colon G(b)\to \mathcal{C}$ to $Y$.
Thus the induced map $(\chi_k)_a\colon\colim_{G(a)}\sigma(a,-)\to Y$ factors through $(\chi_k)_b$. 
Furthermore, the natural transformation $\eta\colon\sigma\to\theta f$, which induces $\colim_{G(c)}\sigma(c,-)\cong \theta(c)$, ensures that this factorization is $(\chi_k)_a = (\chi_k)_b\circ \theta(g)$, which completes the proof that $\chi_k$ is a natural transformation.

To prove the second claim, we need to show that $F^\ast (k,\chi_k) = (k\circ f,f^\ast(\chi_k)\circ\eta)$ equals $(m,\chi_m)$.
Since both $k$ and $m$ are terminal functors, then $k\circ f = m$.
To see that $f^\ast(\chi_k)\circ\eta = \chi_m$, fix an object $(a,\tau)\in \text{Gr}(G)$.
Notice that $(f^\ast(\chi_k)\circ\eta)_{(a,\tau)}$ equals $(\chi_k)_a\circ \eta_{(a,\tau)}$, which is the composition
\begin{center}
\begin{tikzcd}
\sigma(a,\tau) \arrow{r} \arrow[bend right = 15]{rrr}[below]{\eta} &
\colim_{G(a)}{\sigma(a,-)} \arrow{rr}[above]{\text{induced by }\eta}[below]{\cong} & &
\theta(a) \arrow{r}[above]{\chi_k} &
Y.
\end{tikzcd}
\end{center}
But $\chi_k$ was created by inducing maps from the colimit to $Y$ based on $\chi_m$, which means that this composition must also be $(\chi_m)_{(a,\tau)}$.
Therefore, $f^\ast(\chi_k)\circ\eta = \chi_m$ and our second claim has been proven, which finishes the proof.
\end{proof}

\subsection{2-morphisms and homotopical commutivity}

This section is dedicated to showing that a special kind of 2-morphism in $\mathscr{A}_\EuScript{C}$ gives rise to commutivity between homotopy colimits in the homotopy category. We start by recalling some definitions.

Let $\EuScript{M}$ be a category, $I$ be a small category and $D\colon I\to \EuScript{M}$ be a diagram.
The \textit{simplicial replacement} of $D$ is the simplicial object $\text{srep}(D)$ of $\EuScript{M}$ defined by
$$\text{srep}(D)_n = \ds\coprod_{(a_0\leftarrow\dotsm\leftarrow a_n)\in I} D(a_n)$$
where the face map $d_i\colon\text{srep}(D)_n\to \text{srep}(D)_{n-1}$ is induced from the following map on $D(a_n)$ indexed by $(a_0\xleftarrow{\sigma_1}\dotsm\xleftarrow{\sigma_n} a_n)\in I$:
\begin{itemize}
\item for $i=0$, $id\colon D(a_n)\to D(a_n)$ where the codomain is indexed by \\ $(a_1\xleftarrow{\sigma_2} \dotsm \xleftarrow{\sigma_n} a_n)$
\item for $0<i<n$, $id\colon D(a_n)\to D(a_n)$ where the codomain is indexed by \\ $(a_0\xleftarrow{\sigma_1} \dotsm \xleftarrow{\sigma_{i-1}} a_{i-1} \xleftarrow{\sigma_i\sigma_{i+1}}
a_{i+1} \xleftarrow{\sigma_{i+2}} \dotsm \xleftarrow{\sigma_n} a_n)$
\item for $i=n$, $D(\sigma_n)\colon D(a_n)\to D(a_{n-1})$ where the codomain is indexed by \\ $(a_0 \xleftarrow{\sigma_1} \dotsm \xleftarrow{\sigma_{n-1}} a_{n-1})$
\end{itemize}
%
and the degeneracy map $s_i\colon\text{srep}(D)_n\to \text{srep}(D)_{n+1}$ is induced by $id_{D(a_n)}$ where the domain is indexed by $(a_0\xleftarrow{\sigma_1}\dotsm\xleftarrow{\sigma_n} a_n)$ and the codomain is indexed by the chain $(a_0\xleftarrow{\sigma_1} \dotsm \xleftarrow{\sigma_i} a_i \xleftarrow{id} a_i \xleftarrow{\sigma_{i+1}}\dotsm \xleftarrow{\sigma_n} a_n)$.

Additionally suppose that $J$ is a small category and $\alpha\colon J\to I$ is a functor.
Then we can define $\alpha_\#\colon \text{srep}(D\alpha)\to \text{srep}(D)$.
Specifically, $\alpha_\#$ is induced from $id\colon D\alpha(b_n)\to D(\alpha b_n)$ where the domain is indexed by $(b_0\xleftarrow{\chi_1}\dotsm\xleftarrow{\chi_n} b_n)\in J$ and the codomain is indexed by $(\alpha(b_0)\xleftarrow{\alpha(\chi_1)}\dotsm\xleftarrow{\alpha(\chi_n)} \alpha(b_n))\in I$.

Lastly, for any morphism $(\alpha,\eta)\colon(I,D)\to(J,E)$ in $\mathscr{A}_\EuScript{C}$, we get an induced morphism $(\alpha,\eta)_\ast\colon\hocolim{I}{D}\to\hocolim{J}{E}$ given by the composition $|\alpha_\#\circ\eta|$ where $\eta\colon\text{srep}(D)\to\text{srep}(E\alpha)$ is induced in the obvious manner by $(\eta)_i$ for each object $i$ in $I$.
\bigskip

Let $\EuScript{M}$ be a simplicial model category. Suppose that $\theta$ is a 2-morphism in $\mathscr{A}_\EuScript{M}$ from $(\alpha,id)\colon(\EuScript{C},K)\to(\EuScript{D},F)$ to $(\beta,\tau)\colon(\EuScript{C},K)\to(\EuScript{D},F)$ such that $\tau = F\theta$. In particular, this means two things: firstly, we have the following diagram of functors with natural transformation $\theta$
\begin{center}
\begin{tikzcd}[column sep=huge]
\EuScript{C}
  \arrow[bend left=35]{r}[name=U,label=above:$\alpha$]{}
  \arrow[bend right=35]{r}[name=D,label=below:$\beta$]{} &
\EuScript{D}
  \arrow[shorten <=7pt,shorten >=1pt,Rightarrow,to path={(U) -- node[label=right:$\theta$] {} (D)}]{}
  \arrow{r}[above]{F} &
\EuScript{M}
\end{tikzcd}
\end{center}
and secondly, the induced maps $(\alpha,id)_\ast, (\beta,\tau)_\ast\colon\hocolim{\EuScript{C}}{K}\to \hocolim{\EuScript{D}}{F}$ can be written $(\alpha,id)_\ast = |\alpha_\#|$ and $(\beta,\tau)_\ast = |\beta_\#\circ F\theta|$.
The goal of this section is to show that $(\alpha,id)_\ast$ and $(\beta,\tau)_\ast$ commute up to homotopy (using the two particulars mentioned above). To start, we show that $\theta$ gives a ``homotopy'' at the categorical level:

\begin{thm}\label{categorical homotopy}
Let $\EuScript{C}$ and $\EuScript{D}$ be categories, $\EuScript{M}$ be a model category and suppose we have a diagram of functors
\begin{center}
\begin{tikzcd}[column sep=huge]
\EuScript{C}
  \arrow[bend left=35]{r}[name=U,label=above:$\alpha$]{}
  \arrow[bend right=35]{r}[name=D,label=below:$\beta$]{} &
\EuScript{D}
  \arrow[shorten <=7pt,shorten >=1pt,Rightarrow,to path={(U) -- node[label=right:$\theta$] {} (D)}]{}
  \arrow{r}[above]{F} &
\EuScript{M}
\end{tikzcd}
\end{center}
%
where $\theta$ is a natural transformation.
Then there exists a map
\begin{equation}\label{our H}
H\colon (\text{srep}(F\alpha))\times \Delta^1 \rightarrow \text{srep} (F)
\end{equation}
in $s\EuScript{M}$ such that $H_0 = \alpha_\#$ and $H_1 = \beta_\# \circ F\theta$.
\end{thm}

\begin{proof}
Let $I$ be a category with two objects and one nontrivial morphism between them, specifically, the category $[0 \to 1]$. Since $\theta$ is a natural transformation, we get an induced functor $\bar{\theta}\colon \EuScript{C}\times I \to \EuScript{D}$ where $\bar{\theta}(X,0) = \alpha(X)$ and $\bar{\theta} (X,1) = \beta(X)$.

Let $\{1\}$ be the constant simplicial set whose nth level is 1.
Then by inspection, we have the following pushout diagram
\[ \begin{tikzcd}
(\srep F\alpha) \times \{1\} \arrow{r}{i_1} \arrow[swap]{d}{F\theta} & (\srep F\alpha) \times \Delta^1 \arrow{d}{\phi} \\%
(\srep F\beta) \times \{1\} \arrow{r}{j} & \srep F\bar{\theta}
\end{tikzcd}
\]
where $i_1$ is the obvious inclusion map induced from the inclusion $\{1\}\rightarrow \Delta^1$. Notice that $j$ is an inclusion.

By using $\bar{\theta}_\#\colon \srep F\bar{\theta}\to \srep F$, the composition $\bar{\theta}_\# \circ \phi$ is the desired $H$.
\end{proof}

Now we move on to getting a useful cylinder object, which involves some categorical lemmas. We start with some notation.

\begin{defn}\label{new operation}
For an object $Y$, in some category with coproducts $\EuScript{C}$, and a simplicial set $K$, we set
$Y\odot K$ to be the simplicial object of $\EuScript{C}$ whose $n$th level is $(Y\odot K)_n = \coprod_{K_n} Y$ with the obvious morphisms.
\end{defn}

\begin{lemma}\label{real to tensor}
Let $\EuScript{M}$ be a simplicial model category.
If $Y$ is an object of $\EuScript{M}$ and $K$ is a simplicial set, then $|Y\odot K| \cong Y\otimes K$
\end{lemma}

\begin{proof}
Let $Z$ be an object of $\EuScript{M}$. We will show that $\EuScript{M}(|Y\odot K|,Z) \cong \EuScript{M}(Y\otimes K,Z)$ and then by Yoneda's Lemma the result will follow. Let $\Delta$ be the cosimplicial standard simplex. Then
\begin{align*}
\EuScript{M}(|Y\odot K|,Z) &\cong s\EuScript{M}(Y\odot K, Z^\Delta) \\
&\cong sSet(K,\EuScript{M}(Y,Z^\Delta)) \\
&\cong sSet(K,\underline{\text{Map}}(Y,Z)) \\
&\cong \EuScript{M}(Y\otimes K, Z). 
\end{align*}
\end{proof}

\begin{lemma}\label{cyl obj}
Let $\EuScript{M}$ be a simplicial model category with Reedy cofibrant simplicial object $X$. Then $|X\times \Delta^1|$ is a cylinder object for $|X|$, meaning that the folding map $id_{|X|}+id_{|X|}$ factors as $|X|\amalg|X|\to |X\times \Delta^1| \overset{\sim}{\to} |X|$.
\end{lemma}

\begin{proof}
To complete this proof, we need to show two things: $|X\times\Delta^1|$ factors the map $|id|+|id|\colon|X|\coprod|X|\to|X|$ and $|X\times\Delta^1|\simeq|X|$.

First, notice that $id+ id\colon X\coprod X\to X$ factors through $X\times\Delta^1$ in the obvious way. Then, since realization is a left adjoint and hence preserves colimits, the composite
$$|X|\amalg|X| \cong |X\amalg X| \to |X\times\Delta^1| \to |X|$$
is $|id|+|id|$. Thus showing the first condition.

Second, we will look at $|X\times\Delta^1|$. Let $K$ be the bisimplicial object with level $K_{n,m} = \coprod_{\Delta^1_n} X_m$. Notice that $X\times\Delta^1 = \text{diag}(K)$. Thus
$$|X\times\Delta^1| = |\text{diag}(K)| \cong ||K|_{horiz}|_{vert}$$
where the last isomorphism comes from \cite[][Lemma on page 94]{quillen1973higher}.
Furthermore, $|K|_{horiz} = |X|\odot \Delta^1$ and hence
$$|X\times\Delta^1| = ||X|\odot\Delta^1| \cong |X|\otimes \Delta^1$$
by Lemma \ref{real to tensor}. Since $\Delta^1\to\Delta^0$ is a weak equivalence and $|X|$ is cofibrant by \cite[][Proposition 3.6]{goerss2009simplicial}, then 
$$|X|\otimes\Delta^1 \simeq |X|\otimes\Delta^0 = |X|$$
which completes the proof.
\end{proof}
\bigskip

Now we can return to our ``categorical homotopy'' (\ref{our H}). 
We will use Lemma \ref{cyl obj} to prove the following theorem, which will shows that our ``categorical homotopy'' induces a weak equivalence after geometric realization.

\begin{thm}\label{hpty in cat implies hpty in realization}
Let $\EuScript{M}$ be a simplicial model category. If $X$ and $Y$ are simplicial objects in $\EuScript{M}$, $X$ is Reedy cofibrant and there is a morphism $H\colon X\times \Delta^1 \to Y$, then $|H_0|,|H_1|\colon|X|\to|Y|$ are equal in the homotopy category of $\EuScript{M}$.
\end{thm}

\begin{proof}
We will show that $|H_0|$ and $|H_1|$ are left homotopic, which implies that they are equal in the homotopy category of $\EuScript{M}$. Let $\{i\}$ be the constant simplicial object whose nth level is $i$. For $i=0,1$, $H_i$ is the composition
$$X\cong X\times\{i\}\hookrightarrow X\times\Delta^1 \overset{H}{\longrightarrow}\ Y.$$
Thus $|H_i|$ factors through $|H|$ for $i=0,1$. Hence $|H_0|+|H_1|\colon|X|\coprod|X|\to|Y|$ factors through $|H|$. 
Since $|X\times\Delta^1|$ is a cylinder object for $|X|$ (by Lemma \ref{cyl obj}), then the factorization of $|H_0|+|H_1|$ through $|H|$ means that 
$|H_0|+|H_1|$ extends to a map $|X\times\Delta^1|\to |Y|$, i.e.\, $|H_0|$ and $|H_1|$ are left homotopic.
\end{proof}

Finally, we have the desired result of the section:

\begin{cor}\label{homotopical commutivity}
Let $\theta$ be a 2-morphism in $\mathscr{A}_\EuScript{M}$ from $(\alpha,id)\colon(\EuScript{C},K)\to(\EuScript{D},F)$ to $(\beta,\tau)\colon(\EuScript{C},K)\to(\EuScript{D},F)$ such that $\tau = F\theta$. If $\srep(F\alpha)$ is Reedy cofibrant, then the pair $(\alpha,id)_\ast, (\beta,\tau)_\ast\colon\hocolim{\EuScript{C}}{K}\to \hocolim{\EuScript{D}}{F}$ commute up to homotopy.
\end{cor}

\begin{proof}
This is an immediate consequence of Theorem \ref{categorical homotopy} and Theorem \ref{hpty in cat implies hpty in realization}. Specifically, Theorem \ref{categorical homotopy} gives us the necessary morphism $H$, i.e.\ (\ref{our H}), so that we can apply Theorem \ref{hpty in cat implies hpty in realization}. Notice that $|H_0| = |\alpha_\#| = (\alpha,id)_\ast$ and that $|H_1| = |\beta_\#\circ F\theta| = (\beta,\tau)_\ast$.
\end{proof}

\section{Generalized Sieves}\label{section gen sieves}

In this section we define and discuss a particular generalization for a sieve; this will be a key tool in the proofs of Theorems \ref{ucs is a top} and \ref{uhcs is a top} (where we show that certain collections form Grothendieck topologies).
Additionally, we define two special functors. 



\begin{defn}
Fix a positive integer $n$. Let $T_1,\ T_2,\dots,\ T_n$ be sieves on $X$.
A \textit{generalized sieve}, denoted by 
$\prescript{}{X}{[T_1T_2\dots T_n]}$, 
is the following category:
\begin{itemize}

\item objects $(\rho_1,\ \rho_2,\dots,\ \rho_n)$ are $n$-tuples of arrows in $\EuScript{C}$ such that the composition $\rho_1\circ\rho_2\circ\dots\circ\rho_i \in T_i$ 
for all $i=1,\dots,n$.
Pictorially we can visualize this as
\begin{center}
\begin{tikzcd}
X &
A_1 \arrow{l}[above]{\rho_1}[below]{\in T_1} &
A_2 \arrow{l}[above]{\rho_2} \arrow[bend right=30]{ll}[above]{\in T_2} &
\dots\arrow{l}[above]{\rho_3} &
A_n \arrow{l}[above]{\rho_n} \arrow[bend left=30]{llll}[below]{\in T_n}.
\end{tikzcd}
\end{center}

\item morphisms $(f_1,\ f_2,\dots,\ f_n)$ from $(\rho_1,\ \rho_2,\dots,\ \rho_n)$ to $(\tau_1,\ \tau_2,\dots,\ \tau_n)$ are $n$-tuples of arrows in $\EuScript{C}$ where $f_i\colon \text{domain}(\rho_i) \to \text{domain}(\tau_i)$ such that all squares in the following diagram commute
\begin{center}
\begin{tikzcd}
X \arrow{d}[left]{id} & A_1 \arrow{l}[above]{\rho_1} \arrow{d}[left]{f_1} & A_2 \arrow {l}[above]{\rho_2} \arrow{d}[left]{f_2} & \dots \arrow{l}[above]{\rho_3} & A_n \arrow{l}[above]{\rho_n} \arrow{d}[left]{f_n} \\
X & B_1 \arrow{l}[below]{\tau_1} & B_2 \arrow{l}[below]{\tau_2} &\dots \arrow{l}[below]{\tau_3} & B_n \arrow{l}[below]{\tau_n}.
\end{tikzcd}
\end{center}
\end{itemize}
\end{defn}

\noindent For example, if $T$ is a sieve on $X$, then $\prescript{}{X}{[T]}$ is $T$ (as categories). 

\begin{rmk}\label{groth construct}
For sieves $T_1,\dots,T_n$ on $X$ we can define a functor
$$G\colon\prescript{}{X}{[T_1T_2\dots T_{n-1}]}\to Cat,\quad (\rho_1,\dots,\rho_{n-1})\mapsto (\rho_1\circ\dots\circ\rho_{n-1})^\ast T_n.$$
Then the Grothendieck construction for $G$ is $\prescript{}{X}{[T_1T_2\dots T_n]}$.
Indeed, this is easy to see once we view the objects of $\prescript{}{X}{[T_1T_2\dots T_n]}$ as pairs $$((\rho_1,\dots,\rho_{n-1})\in\prescript{}{X}{[T_1\dots T_{n-1}]}, \tau\in G(\rho_1,\dots,\rho_{n-1})).$$

\end{rmk}

Like a sieve, a generalized sieve $\prescript{}{X}{[T_1\dots T_n]}$ can be viewed as a subcategory of $(\EuScript{C}\downarrow X)$.
Thus we will use $U$ (see Notation \ref{forgetful functor}) as the functor $\prescript{}{X}{[T_1T_2\dots T_n]} \to \EuScript{C}$ given by $(\rho_1, \rho_2, \dots, \rho_n) \mapsto \text{domain}\ \rho_n$.
Note: for any morphism $(f_1,\ f_2,\dots,\ f_n)$, $U(f_1,\ f_2,\dots,\ f_n) = f_n$.

\begin{defn}
Let $T_1,\ T_2,\dots,\ T_n$ be sieves on $X$ (with $n\geq 2$), we define a `forgetful functor'
$$\mathscr{F}\colon \prescript{}{X}{[T_1T_2\dots T_n]} \to \prescript{}{X}{[T_1T_2\dots T_{n-1}]},\quad (\rho_1, \rho_2, \dots, \rho_n) \mapsto (\rho_1, \rho_2, \dots, \rho_{n-1}).$$
Pictorially,
\begin{center}
\begin{tikzcd}
X &
A_1 \arrow{l}[above]{\rho_1}&
A_2 \arrow{l}[above]{\rho_2} & 
\dots\arrow{l}[above]{\rho_3} &
A_{n-1} \arrow{l}[above]{\rho_{n-1}} &
A_n \arrow{l}[above]{\rho_n} \\ 
{} &
{} &
X \arrow[Mapsfrom]{l}[above]{\mathscr{F}} &
A_1 \arrow{l}[above]{\rho_1}&
A_2 \arrow{l}[above]{\rho_2} &
\dots\arrow{l}[above]{\rho_3} &
A_{n-1} \arrow{l}[above]{\rho_{n-1}}. 
\end{tikzcd}
\end{center}
\end{defn}

\begin{rmk}
Actually, the above definition only needs $n\geq 1$.
In the $n=1$ case, our forgetful functor is $\mathscr{F}\colon \prescript{}{X}{[T_1]}\to \prescript{}{X}{[\,]}$, where $\prescript{}{X}{[\,]}$ is the category with unique object $(id_X\colon X\to X)$ and no non-identity morphisms, and is defined by $\rho\mapsto id_X$.
\end{rmk}

Now we take this functor $\mathscr{F}$ and use it to make an arrow in $\mathscr{A}_\EuScript{C}$:

\begin{defn}
For any sieves $T_1,\ T_2,\dots,\ T_n$ on $X$ (with $n\geq 2$), define a map in $\mathscr{A}_\EuScript{C}$ called
$\widetilde{\mathscr{F}}\colon \left( \prescript{}{X}{[T_1T_2\dots T_n]}, U\right) \to \left( \prescript{}{X}{[T_1T_2\dots T_{n-1}]}, U\right)$
by
$\widetilde{\mathscr{F}} = (\mathscr{F}, \eta_{\mathscr{F}})$
where
$\eta_{\mathscr{F}}\colon U\to (U\circ\mathscr{F})$
is given by
$(\eta_{\mathscr{F}})_{(\rho_1, \rho_2, \dots, \rho_n)} = \rho_n$.
\end{defn}

The fact that $ \eta_{\mathscr{F}}$ is a natural transformation can be seen easily from the pictorial view of morphisms.
Specifically, consider the morphism $(f_1, f_2, \dots, f_n)$; this morphism gives us a commutative diagram
\begin{center}
\begin{tikzcd}
X \arrow{d}[left]{id} &
A_1 \arrow{l}[above]{\rho_1} \arrow{d}[left]{f_1} &
A_2 \arrow {l}[above]{\rho_2} \arrow{d}[left]{f_2} &
\dots \arrow{l}[above]{\rho_3} &
A_{n-1} \arrow{l}[above]{\rho_{n-1}} \arrow{d}[left]{f_{n-1}} &
A_n \arrow{l}[above]{\rho_n} \arrow{d}[left]{f_n} \\
X &
B_1 \arrow{l}[below]{\tau_1} &
B_2 \arrow{l}[below]{\tau_2} &
\dots \arrow{l}[below]{\tau_3} &
B_{n-1} \arrow{l}[above]{\tau_{n-1}} &
B_n \arrow{l}[below]{\tau_n}
\end{tikzcd}
\end{center}
but the rightmost commutative square of the above diagram can be relabelled to give us the following commutative diagram
\begin{center}
\begin{tikzcd}
U\circ\mathscr{F}(\rho_1, \rho_2, \dots, \rho_n) \arrow{d}[left]{U\circ\mathscr{F}(f_1,f_2,\dots,f_n)} &
U(\rho_1, \rho_2, \dots, \rho_n) \arrow{l}[above]{\eta_{\mathscr{F}}} \arrow{d}[left]{U(f_1,f_2,\dots,f_n)}\\
U\circ\mathscr{F}(\tau_1, \tau_2, \dots, \tau_n)&
U(\tau_1, \tau_2, \dots, \tau_n) \arrow{l}[below]{\eta_{\mathscr{F}}}
\end{tikzcd}
\end{center}
and it is this diagram that shows $\eta_{\mathscr{F}}$ is a natural transformation.

\begin{defn}
Let $T_1,\ T_2,\dots,\ T_n$ be sieves on $X$ (with $n\geq 2$), we define a `composition functor'
$$\mu\colon \prescript{}{X}{[T_1T_2\dots T_n]} \to \prescript{}{X}{[T_2\dots T_{n}]}, \quad (\rho_1, \rho_2, \dots, \rho_n) \mapsto (\rho_1\circ\rho_2, \rho_3, \dots, \rho_{n}).$$
Pictorially,
\begin{center}
\begin{tikzcd}
X &
A_1 \arrow{l}[above]{\rho_1}&
A_2 \arrow{l}[above]{\rho_2} &
\dots\arrow{l}[above]{\rho_3} &
A_{n-1} \arrow{l}[above]{\rho_{n-1}} &
A_n \arrow{l}[above]{\rho_n} \\
{} &
{} &
X \arrow[Mapsfrom]{l}[above]{\mu} &
A_2 \arrow{l}[above]{\rho_1\circ\rho_2}&
A_3 \arrow{l}[above]{\rho_3} &
\dots\arrow{l}[above]{\rho_4} &
A_{n} \arrow{l}[above]{\rho_{n}}. 
\end{tikzcd}
\end{center}
\end{defn}

Now we take this functor $\mu$ and use it to make an arrow in $\mathscr{A}_\EuScript{C}$:

\begin{defn}
For any sieves $T_1,\ T_2,\dots,\ T_n$ on $X$ (with $n\geq 2$).
Define
$\widetilde{\mu}\colon \left( \prescript{}{X}{[T_1T_2\dots T_n]}, U\right) \to \left( \prescript{}{X}{[T_2T_3\dots T_{n}]}, U\right)$
by
$\widetilde{\mu} = (\mu, \eta_{\mu})$
where the natural transformation 
$\eta_{\mu}\colon U\to (U\circ\mu)$
is given by
$(\eta_{\mu})_{(\rho_1, \rho_2, \dots, \rho_n)} = id_{\text{domain}\ \rho_n}$.
\end{defn}



Lastly, we include an two results. 

\begin{cor}\label{forgetful maps induce isom}
Let $V$ and $W$ be sieves on $X$ such that for all $f\in V$, $f^\ast W$ is a colim sieve. Fix an integer $n\geq 0$ and let $T_1,\ T_2,\ \dots,\ T_n$ be a list of sieves on $X$ (note: $n=0$ corresponds to the empty list). Then the induced map $\widetilde{\mathscr{F}}^\ast\colon  \mathscr{A}_\EuScript{C} (\prescript{}{X}{[T_1T_2\dots T_nV]}, cY) \to \mathscr{A}_\EuScript{C} (\prescript{}{X}{[T_1T_2\dots T_nVW]}, cY)$ is a bijection for all objects $Y$ of $\EuScript{C}$.
\end{cor}

\begin{proof}
This is an immediate application of Proposition \ref{technical lemma} and Remark \ref{groth construct}.

\end{proof}

\begin{lemma}\label{vert maps are we in diagram}
Let $n\geq 1$ and $T_1,\dots,T_n$ be sieves on $X$ such that for all $f\in T_{n-1}$, $f^\ast T_{n}$ is a universal hocolim sieve.
Then the induced map
$$\mathscr{F}_\ast\colon \hocolim{\prescript{}{X}{[T_1\dots T_n]}}{U} \to \hocolim{\prescript{}{X}{[T_1\dots T_{n-1}]}}{U}$$
%
is a weak equivalence. 
Note: when $n=1$, then $T_{n-1} = \{id_X\colon X\to X\}$ and
$\prescript{}{X}{[T_1\dots T_{n-1}]} = \prescript{}{X}{[\ ]}$.
\end{lemma}

\begin{proof}
We will use $\rho$ as an abbreviation for $(\rho_1,\dots,\rho_{n-1})\in \prescript{}{X}{[T_1\dots T_{n-1}]}$.
Additionally, we will abuse notation and use $\rho$ to represent $\rho_1\circ\dots\circ\rho_{n-1}$ (e.g. $\rho^\ast T_n$).

By remark \ref{groth construct}, $\prescript{}{X}{[T_1\dots T_{n}]}$ is a Grothendieck construction and its objects are $(\rho\in\prescript{}{X}{[T_1\dots T_{n-1}]}, \tau\in \rho^\ast T_n)$.
Thus by \cite[][Theorem 26.8]{GrothendieckConstruction},
$$\hocolim{\prescript{}{X}{[T_1\dots T_{n}]}}{U} \simeq
\hocolim{\rho\in\prescript{}{X}{[T_1\dots T_{n-1}]}}{\hocolim{\rho^\ast T_n}{U}}.$$

On the other hand, by assumption, for all $\rho\in \prescript{}{X}{[T_1\dots T_{n-1}]}$,
$$ \hocolim{\rho^\ast T_n}{U} \simeq \text{domain }(\rho). $$
%
Thus
$$\hocolim{\rho\in\prescript{}{X}{[T_1\dots T_{n-1}]}}{\hocolim{\rho^\ast T_n}{U}}
\simeq \hocolim{\prescript{}{X}{[T_1\dots T_{n-1}]}}{U} .$$
%
Putting everything together yields $\hocolim{\prescript{}{X}{[T_1\dots T_{n}]}}{U} \simeq \hocolim{\prescript{}{X}{[T_1\dots T_{n-1}]}}{U}$ and therefore $\mathscr{F}_\ast$ is a weak equivalence.
\end{proof}

\section{Universal Colim and Hocolim Sieves}\label{UCS and HCS form a GTop}

In this section we show that the collections of universal colim sieves and universal hocolim sieves form Grothendieck topologies. 
As we will see later, the maximality and stability conditions follow easily, so we will focus our discussion on the transitivity condition.

Let $\EuScript{U}$ be either the collection of universal colim sieves or the collection of hocolim sieves for the category $\EuScript{C}$ with $\EuScript{U}(X)$ the universal colim/hocolim sieves on $X$.
From here on out, we fix $S\in\mathcal{U}(X)$ and a sieve $R$ on $X$ such that for all $f\in S$, $f^\ast R\in\EuScript{U}(\text{domain}\ f)$.
We want to prove that $R\in\EuScript{U}(X)$.
We will specifically discuss our technique for showing that $R$ is a colim/hocolim sieve; universality is not difficult to see and will be shown later.

\begin{rmk}\label{rework colimit into IFcat}
By definition, $R$ is a colim sieve if and only if $X$ is a colimit for $R$.
But by Lemma \ref{reword colim defn}, this is equivalent to the induced map $\phi_R^\ast$, specifically $\phi_R^\ast\colon \mathscr{A}_\EuScript{C}(cX,cY) \to \mathscr{A}_\EuScript{C}(R,cY)$, being a bijection for all objects $Y$ of $\EuScript{C}$ (see Notation \ref{cocone sieve map} for the definition of $\phi_R$).
\end{rmk}
\bigskip

\noindent\textsc{General Outline for Transitivity}
\bigskip

We will be using the following noncommutative diagram in $\mathscr{A}_\EuScript{C}$:
\begin{equation}\label{diagram 1}
\begin{tikzcd}
R \arrow{r}[above]{\phi_R} &
cX \\
\prescript{}{X}{[RS]} \arrow{u}[left]{\widetilde{\mathscr{F}}} &
S \arrow{u}[right]{\phi_S} \\
\prescript{}{X}{[RSR]} \arrow{u}[left]{\widetilde{\mathscr{F}}} \arrow{r}[below]{\widetilde{\mu}} &
\prescript{}{X}{[SR]} \arrow{uul}[above right]{\widetilde{\mu}} \arrow{u}[right] {\widetilde{\mathscr{F}}}.
\end{tikzcd}
\end{equation}
Note: $\prescript{}{X}{[T_1T_2\dots T_n]}$ is shorthand for $( \prescript{}{X}{[T_1T_2\dots T_n]}, U)$, just like how $R$ and $S$ are shorthand for $(R,U)$ and $(S,U)$ respectively.

\begin{itemize}
\item We will show that the upper right triangle commutes and the lower left triangle commutes up to a 2-morphism.
\item Then we will work with the two cases: (i) universal colim sieves, (ii) universal hocolim sieves.
\begin{enumerate}
  \item[(i)] We will apply $\mathscr{A}_\EuScript{C}(-,cY)$ levelwise to the diagram.
    \begin{itemize}
    \item By Lemma \ref{ll triangle commutes} this will result in a commutative diagram.
    \item By Corollary \ref{forgetful maps induce isom} all resulting vertical maps will be bijections.
    \end{itemize}
  \item[(ii)] We will apply homotopy colimits levelwise to the diagram.
    \begin{itemize}
    \item By Corollary \ref{homotopical commutivity} this will result in a commutative diagram.
    \item By Lemma \ref{vert maps are we in diagram} all resulting vertical maps will be weak equivalences.
    \end{itemize}
\end{enumerate}
\item It will then follow formally that the map induced by $\phi_R$ is a bijection/weak equivalence (depending on the case).
\end{itemize}

Since the first piece of this outline depends solely on diagram (\ref{diagram 1}), we discuss it now; the rest of the outline will be completed during the proofs of Theorems \ref{ucs is a top} and \ref{uhcs is a top} where we show that the collections of universal colim sieves and universal hocolim sieves form Grothendieck topologies.

\bigskip

\pagebreak

\noindent\textsc{``Commutivity'' of diagram (\ref{diagram 1})}
\bigskip

\begin{lemma}\label{ur triangle lemma}
In diagram (\ref{diagram 1}), the upper right triangle commutes.
\end{lemma}

\begin{proof}
We start by unpacking what the compositions in the diagram are:
$$ \phi_R\circ\widetilde{\mu} =
(t,\varphi_R) \circ (\mu,\eta_{\mu}) =
(t\circ\mu, \mu^\ast\varphi_R \circ\eta_{\mu}) $$
$$ \phi_S\circ\widetilde{\mathscr{F}} =
(t,\varphi_S) \circ (\mathscr{F},\eta_{\mathscr{F}}) =
(t\circ\mathscr{F}, \mathscr{F}^\ast\varphi_S \circ\eta_{\mathscr{F}}) $$
Since $t$ is the terminal map, then
$ t\circ\mu = t\circ\mathscr{F}$.
To see that the natural transformations 
are the same fix
$(\rho,\tau)\in\prescript{}{X}{[SR]}$.
Then
$$ (\mu^\ast\varphi_R \circ\eta_{\mu})_{(\rho,\tau)} =
(\varphi_R)_{\mu(\rho,\tau)} \circ id =
(\varphi_R)_{\rho\circ\tau} =
\rho\circ\tau $$
%
and
$$ (\mathscr{F}^\ast\varphi_S \circ\eta_{\mathscr{F}})_{(\rho,\tau)} =
(\varphi_S)_{\mathscr{F}(\rho,\tau)} \circ \tau =
(\varphi_S)_{\rho} \circ \tau =
\rho\circ\tau. $$
Since the natural transformations are the same on all objects, the proof is complete.
\end{proof}

At this point it would be nice if the lower left triangle in the diagram also commuted, however, it does not.
Instead, it contains a 2-morphism:

\begin{lemma}\label{ll triangle 2 morph}
There exists a 2-morphism $\theta\colon \widetilde{\mu}\circ\widetilde{\mu}\to \widetilde{\mathscr{F}}\circ\widetilde{\mathscr{F}}$ where \\ $\widetilde{\mu}\circ\widetilde{\mu},\widetilde{\mathscr{F}}\circ\widetilde{\mathscr{F}}\colon \prescript{}{X}{[RSR]}\to R$.
\end{lemma}

Two remarks: First, $\prescript{}{X}{[R]} = R$.
Second, this lemma and (a similar) proof hold for $\prescript{}{X}{[T_1T_2\dots T_n]}\to\prescript{}{X}{[T_1T_2\dots T_{n-2}]}$ when all $T_{odd} = T_1$ and $T_{even} = T_2$.
The two morphisms ``are'' $\mu\circ\mu\colon (\rho_1,\dots,\rho_n)\mapsto (\rho_1\circ\rho_2\circ\rho_3,\rho_4,\dots,\rho_n)$ and $\mathscr{F}\circ\mathscr{F}\colon (\rho_1,\dots,\rho_n)\mapsto (\rho_1,\dots,\rho_{n-2})$.

\begin{proof}
We start by recalling $\mu\circ\mu\colon \left[X\xleftarrow{\rho} A\xleftarrow{\tau} B\xleftarrow{\gamma} C\right]\mapsto\left[X\xleftarrow{\rho\tau\gamma} C\right]$ and \\ $\mathscr{F}\circ\mathscr{F}\colon\left[X\xleftarrow{\rho} A\xleftarrow{\tau} B\xleftarrow{\gamma} C\right]\mapsto\left[X\xleftarrow{\rho} A\right]$.
Now define $\theta\colon \widetilde{\mu}\circ\widetilde{\mu} \to \widetilde{\mathscr{F}}\circ\widetilde{\mathscr{F}}$ by
$(\theta)_{(\rho,\tau ,\gamma)} =
\tau\circ\gamma$.
We claim that this $\theta$ is the desired 2-morphism.

First, $\theta$ is clearly a natural transformation from $\mu^2$ to $\mathscr{F}^2$. Indeed, consider the following object in $\prescript{}{X}{[RSR]}$:
\begin{center}
\begin{tikzcd}
X &
A \arrow{l}[below]{\rho}[above]{\in R} &
B \arrow{l}[below]{\tau}\arrow[bend right = 30]{ll}[above]{\in S} &
C \arrow{l}[below]{\gamma}\arrow[bend right = 45]{lll}[above]{\in R}.
\end{tikzcd}
\end{center}
Notice that $\theta$ does the correct thing on objects since
$\mu^2(\rho,\tau,\gamma) =$
\begin{tikzcd}
X &
C \arrow{l}[below]{\rho\circ\tau\circ\gamma}[above]{\in R}
\end{tikzcd}
and
$\mathscr{F}^2(\rho,\tau,\gamma) =$
\begin{tikzcd}
X &
A \arrow{l}[below]{\rho}[above]{\in R},
\end{tikzcd}
and thus $\theta_{(\rho,\tau,\gamma)} = \tau\circ\gamma\colon C\to A$ is a morphism from $\mu^2 (\rho,\tau,\gamma)$ to $\mathscr{F}^2 (\rho,\tau,\gamma)$ in $R$.
It is similarly easy to see that $\theta$ behaves compatibly with the morphisms of $\prescript{}{X}{[RSR]}$.

Second, fix $(\rho,\tau,\gamma)\in\prescript{}{X}{[RSR]}$.
We also need to know that the diagram
\begin{center}
\begin{tikzcd}
& \text{domain}(\gamma) \arrow{dl}[above left]{id}\arrow{dr}[above right]{\tau\circ\gamma} & \\
\text{domain}(\rho\circ\tau\circ\gamma) \arrow{rr}[below]{\theta\ =\ \tau\circ\gamma} & &
\text{domain}(\rho)
\end{tikzcd}
\end{center}
is commutative, which it clearly is. Therefore, $\theta$ is our desired 2-morphism.
\end{proof}
\bigskip

\noindent\textsc{Grothendieck Topologies}
\bigskip

\begin{thm}\label{ucs is a top}
Let $\EuScript{C}$ be any category. The collection of all universal colim sieves on $\EuScript{C}$ forms a Grothendieck topology.
\end{thm}

\begin{proof}
Let $\EuScript{U}$ be the collection of universal colim sieves for the category $\EuScript{C}$ with $\EuScript{U}(X)$ the collection of universal colim sieves on $X$.
The first two properties, i.e. the maximal and stability axioms, are easy to check.
Indeed, stability is immediate from the definition of universal colim sieve whereas
the maximal sieve on $X$ is the category $(\EuScript{C}\downarrow X)$, which
has a terminal object, namely $id\colon X\to X$.
Thus the inclusion functor $L\colon \ast\to (\EuScript{C}\downarrow X)$ given by $L(\ast) = id$ (see Notation \ref{trivial cat note}) is a final functor.
Hence by \cite[][Theorem 1, Section 3, Chapter IX]{cwm}
$$\colim_{(\EuScript{C}\downarrow X)}{U} \cong \colim_{\ast}{UL} \cong UL(\ast) = X$$
and so the maximal sieve on $X$ is a colim sieve.
Moreover, for all $f\colon Y\to X$ in $\EuScript{C}$, $f^\ast (\EuScript{C}\downarrow X) = (\EuScript{C}\downarrow Y)$, which by the previous argument is a colim sieve on $Y$.
Therefore, $(\EuScript{C}\downarrow X)\in\EuScript{U}(X)$.

In order to prove transitivity, we fix $S\in\EuScript{U}(X)$ and a sieve $R$ on $X$ such that for all $f\in S$, $f^\ast R\in\EuScript{U}(\text{domain}\ f)$.
We need to prove that $R\in\EuScript{U}(X)$. First we will remove the need to show universality.
Indeed, up to notation, for any morphism $\alpha$ in $\EuScript{C}$ with codomain $X$, we have the same assumptions for $\alpha^\ast R$ as we have for $R$ (when we use $\alpha^\ast S$ instead of $S$). 
In particular, this means that showing $R$ is a colim sieve on $X$ will also show (up to notation) that each $\alpha^\ast R$ is a colim sieve. 
Therefore it suffices to show that $R$ is a colim sieve.
But by Remark \ref{rework colimit into IFcat} this means: to prove that $R$ is a universal colim sieve, it suffices to prove that $\phi_R^\ast\colon  \mathscr{A}_\EuScript{C}(cX,cY) \to \mathscr{A}_\EuScript{C}(R,cY)$ is a bijection for all objects $Y$ of $\EuScript{C}$.

Now fix $Y$, an object of $\EuScript{C}$, and apply $\mathscr{A}_\EuScript{C}(-,cY)$ to diagram (\ref{diagram 1}) in order to obtain the following 
diagram of sets:
\begin{equation}\label{diagram 2}
\begin{tikzcd}
\mathscr{A}_\EuScript{C}(R,cY) \arrow{d}[left]{\widetilde{\mathscr{F}}^\ast}\arrow{ddr}[above right]{\widetilde{\mu}^\ast}&
\mathscr{A}_\EuScript{C}(cX,cY) \arrow{l}[above]{\phi_R^\ast}\arrow{d}[right]{\phi_S^\ast} \\
\mathscr{A}_\EuScript{C}(\prescript{}{X}{[RS]},cY) \arrow{d}[left]{\widetilde{\mathscr{F}}^\ast} &
\mathscr{A}_\EuScript{C}(S,cY)\arrow{d}[right]{\widetilde{\mathscr{F}}^\ast} \\
\mathscr{A}_\EuScript{C}(\prescript{}{X}{[RSR]},cY) &
\mathscr{A}_\EuScript{C}(\prescript{}{X}{[SR]},cY) \arrow{l}[below] {\widetilde{\mu}^\ast}.
\end{tikzcd}
\end{equation}
We will use this diagram to prove that $\phi_R^\ast$ is a bijection.

The upper right triangle in diagram (\ref{diagram 2}) commutes by Lemma \ref{ur triangle lemma}.
Moreover, since the lower left triangle in the first diagram contained a 2-morphism (by Lemma \ref{ll triangle 2 morph}), 
then Lemma \ref{ll triangle commutes} shows that the lower left triangle in diagram (\ref{diagram 2}) commutes.
Thus (\ref{diagram 2}) is a commutative diagram of sets.

Now we will discuss some of the morphisms in (\ref{diagram 2}).
First, notice that by Lemma \ref{reword colim defn}, since $S$ is a colim sieve, $\phi_S^\ast$ is a bijection.
Second, notice that Corollary \ref{forgetful maps induce isom} implies that all of the maps $\widetilde{\mathscr{F}}^\ast$ in diagram (\ref{diagram 2}) are bijections.
Indeed, by Corollary \ref{forgetful maps induce isom}, our assumptions on $R$ imply that both induced morphisms $\widetilde{\mathscr{F}}^\ast\colon \mathscr{A}_\EuScript{C}(S,cY) \to \mathscr{A}_\EuScript{C}(\prescript{}{X}{[SR]},cY)$ and $\widetilde{\mathscr{F}}^\ast\colon \mathscr{A}_\EuScript{C}(\prescript{}{X}{[RS]},cY) \to \mathscr{A}_\EuScript{C}(\prescript{}{X}{[RSR]},cY)$ are bijections, and our assumptions on $S$ imply that the induced morphism $\widetilde{\mathscr{F}}^\ast\colon \mathscr{A}_\EuScript{C}(R,cY) \to \mathscr{A}_\EuScript{C}(\prescript{}{X}{[RS]},cY)$ is a bijection.
Hence all vertical maps in diagram (\ref{diagram 2}) are isomorphisms.

We summarize the results about diagram (\ref{diagram 2}): 
we have commutative triangles that combine to make a commutative diagram of sets of the form
\begin{equation}\label{diagram 3}
\begin{tikzcd}
\mathscr{A}_\EuScript{C}(R,cY) \arrow{d}[left]{\cong} \arrow{dr}[below left]{\alpha}&
\mathscr{A}_\EuScript{C}(cX,cY)  \arrow{l}[above]{\phi_R^\ast}\arrow{d}[right]{\cong} \\
A &
B \arrow{l}.
\end{tikzcd}
\end{equation}
Notice that some of the details mentioned in diagram (\ref{diagram 2}) are not mentioned in the above diagram. Indeed, we only need to know that for each $Y$ some such $A$, $B$ and $\alpha$ exist, their specific values are not required; diagram (\ref{diagram 2}) is what guarantees their existance.

Using the lower left triangle in diagram (\ref{diagram 3}) we see that $\alpha$ is an injection.
Whereas the upper right triangle in diagram (\ref{diagram 3}) shows that $\alpha$ is a surjection.
Therefore, $\alpha$ is a bijection.
Now the commutativity of the upper right triangle in diagram (\ref{diagram 3}) implies that $\phi_R^\ast$ is a bijection. 
Hence we have completed the proof of transitivity.
\end{proof}


\begin{thm}\label{uhcs is a top}
For a simplicial model category $\EuScript{M}$, the collection of all universal hocolim sieves on $\EuScript{M}$ forms a Grothendieck topology, which we dub the \textit{homotopical canonical topology}.
\end{thm}

\begin{proof}
Let $\EuScript{U}$ be the collection of universal hocolim sieves for the simplicial model category $\EuScript{M}$ with $\EuScript{U}(X)$ the collection of universal hocolim sieves on $X$.
The first two conditions of a Grothendieck topology are easy to check.
Indeed, stabiility automatically follows from the definition of universal hocolim sieve whereas maximality follows from $f^\ast (\EuScript{M}\downarrow X) = (\EuScript{M}\downarrow Y)$. Specifically, for all $f\colon Y\to X$, $f^\ast (\EuScript{M}\downarrow X) = (\EuScript{M}\downarrow Y)$ and thus in order to prove the first condition, it suffices to show $\hocolim{(\EuScript{M}\downarrow X)} U \simeq X$. But $(\EuScript{M}\downarrow X)$ has a final object, namely $X \stackrel{\text{id}}{\rightarrow} X$.
But by \cite[][Section 6, Lemma 6.8]{primer}, $$\hocolim{(\EuScript{M}\downarrow X)} U \simeq U(id) = X.$$

The rest of the proof will focus on transitivity.
Fix a sieve $S\in\EuScript{U}(X)$ and a sieve $R$ on $X$ such that for all $f\in S$, $f^\ast R\in\EuScript{U}(\text{domain}\ f)$.
We will show that $R\in\EuScript{U}(X)$.

We start by removing the need to show universality.
Up to notation, for any morphism $\alpha$ in $\EuScript{M}$ with codomain $X$, we have the same assumptions for $\alpha^\ast R$ as we have for $R$ (when we use $\alpha^\ast S$ instead of $S$).
In particular, this means that showing $R$ is a hocolim sieve on $X$ will also show (up to notation) that each $\alpha^\ast R$ is a hocolim sieve.
Therefore it suffices to show that $R$ is a hocolim sieve.

Now take diagram (\ref{diagram 1}) and apply homotopy colimits levelwise to obtain the following noncommutative diagram:
\begin{equation}\label{ho diagram 2}
\begin{tikzcd}
\hocolim{\prescript{}{X}{[R]}}{U}
\arrow{r}[above]{\mathscr{F}_\ast} &
\hocolim{\prescript{}{X}{[\ ]}}{U}
\arrow{r}[above]{\simeq} &
X\\
\hocolim{\prescript{}{X}{[RSR]}}{U}
\arrow{u}[left]{(\mathscr{F}\circ\mathscr{F})_\ast}
\arrow{r}[below]{\mu_\ast} &
\hocolim{\prescript{}{X}{[SR]}}{U}
\arrow{u}[right]{(\mathscr{F}\circ\mathscr{F})_\ast}
\arrow{ul}[above right]{\mu_\ast}.
\end{tikzcd}
\end{equation}
Remark:
In the above diagram, we think of $cX$ as $\prescript{}{X}{[\,]}$, the subcategory of $(\EuScript{M} \downarrow X)$ containing $(id_X\colon X \to X)$ as its only object and no non-identity morphisms, which allows us to write $\phi_S$ as $\mathscr{F}$.

Since $\prescript{}{X}{[R]} = R$, then we can prove that $R$ is a hocolim sieve on $X$ by showing that the top horizontal map $\mathscr{F}_\ast$ in (\ref{ho diagram 2}) is a weak equivalence.

First notice that all vertical maps $(\mathscr{F}\circ\mathscr{F})_\ast$ in (\ref{ho diagram 2}) are weak equivalences since $(\mathscr{F}\circ\mathscr{F})_\ast = \mathscr{F}_\ast \circ \mathscr{F}_\ast$ and by Lemma \ref{vert maps are we in diagram}. Second notice that by Lemma \ref{ll triangle 2 morph} and the Reedy cofibrancy of $\srep{(U\mu^2)}$, we may apply Corollary \ref{homotopical commutivity}. Hence every part of diagram (\ref{ho diagram 2}) commutes up to homotopy.

We now summarize the discussion from earlier in the section by summarizing the pertinent results about diagram (\ref{ho diagram 2}): 
in the homotopy category, we have commutative triangles that combine to make a commutative diagram of the form
\begin{center}
\begin{tikzcd}
\hocolim{\prescript{}{X}{[R]}}{U}
\arrow{r}[above]{\mathscr{F}_\ast} &
\hocolim{\prescript{}{X}{[\ ]}}{U}
\arrow{r}[above]{\cong} &
X\\
A
\arrow{u}[left]{\cong}
\arrow{r} &
B
\arrow{u}[right]{\cong}
\arrow{ul}.
\end{tikzcd}
\end{center}
By applying $\text{Ho}_\EuScript{M}(Z,-)$ (i.e.\ the homotopy classes of maps in $\EuScript{M}$ from $Z$ to $-$) levelwise to the above diagram, it follows immediately that 
the diagonal morphism $d_Z\colon \text{Ho}_\EuScript{M}(Z,B) \to\text{Ho}_\EuScript{M}(Z,\hocolim{\prescript{}{X}{[R]}}{U})$ is a bijection.
Indeed, the two ways to get from $B$ to $X$ imply that $d_Z$ is an injection whereas the two ways to get from $A$ to $\hocolim{\prescript{}{X}{[R]}}{U}$ imply that $d_Z$ is a surjection.
Since $d_Z$ is a bijection for all $Z$, then the diagonal map $B\to \hocolim{\prescript{}{X}{[R]}}{U}$ is an isomorphism.
Thus the diagram's commutativity implies that the top horizontal morphism $\mathscr{F}_\ast$ is also an isomorphism. 
Hence we have completed the proof of transitivity.
\end{proof}

\section{Universal Colim Sieves and the Canonical Topology}\label{section ucs and can top}

In this section we show that the collection of all universal colim sieves forms the canonical topology; this folklore result is mentioned in \cite{johnstone2002sketches}.
Additionally, we give a basis for the canonical topology.



\begin{thm}\label{can top is ucs}
For any (locally small) category $\EuScript{C}$, the collection of all universal colim sieves on $\EuScript{C}$ is the canonical topology.
\end{thm}


\begin{proof}
We start with a fact that will be used a few times: The equalizer in the sheaf condition can be expressed as a limit over a covering sieve. Specifically, for a presheaf $F$ and covering sieve $S$ 
\begin{equation}\label{star}
\Eq\left(\prod_{A\xrightarrow{f} X\in S} F(A) \substack{\xrightarrow{\alpha} \\ \xrightarrow[\beta]{}} \prod_{\substack{B\xrightarrow{g} A \\ A\xrightarrow{f} X\in S}} F(B)\right) = \lim_{\substack{\longleftarrow \\ S}} FU
\end{equation}
where the $fg$ component of $\alpha((x_f)_{f\in S})$ is $x_{fg}$ and of $\beta((x_f)_{f\in S})$ is $Fg(x_{f})$ \cite[see][Theorem 2, Section 2, Chapter V]{cwm}.

Let $\EuScript{U}$ be the universal colim sieve topology for the category $\EuScript{C}$ with $\EuScript{U}(X)$ the collection of universal colim sieves on $X$. In a similar vein, let $C$ be the canonical topology for $\EuScript{C}$. 
Let $rM$ denoted the representable presheaf on $M$, i.e. for all objects $K$ of $\EuScript{C}$, $rM(K) = \EuScript{C}(K,M)$.
We will show that the universal colim sieves form a ``larger topology'' than the canonical topology, i.e. $C(X) \subset \EuScript{U}(X)$ for all objects $X$, 
and that $\EuScript{U}$ is subcanonical, i.e. that $\EuScript{U}$ is a topology where all representable presheaves are sheaves.
This will prove the desired result because the canonical topology is the largest subcanonical topology.

To see that $C(X)\subset \EuScript{U}(X)$, let $S\in C(X)$, 
$f\colon Y\to X$ be a morphism 
and $M$ be an object in $\EuScript{C}$.
Since $f^\ast S\in C(Y)$ and $rM$ is a sheaf in the canonical topology, then it follows from the the sheaf condition and (\ref{star}) that
\begin{equation*}
rM(Y) \cong \lim_{\substack{\longleftarrow \\ f^\ast S}} (rM\circ U).
\end{equation*}
Thus by rewriting what $rM(-)$ means, we get
$$\EuScript{C}(Y,M)\cong \ds\lim_{g \in f^\ast S}\EuScript{C}\left(U(g), M \right)$$
for every object $M$.
This formally implies that $\colim_{f^\ast S}U$ exists and
$$\EuScript{C}(Y,M)\cong \EuScript{C}\left(\ds\colim_{f^\ast S} U, M \right)$$
for all objects $M$ of $\EuScript{C}$.
Now by Yoneda's Lemma, $Y\cong \colim_{f^\ast S} U$, i.e. $f^\ast S$ is a colim sieve. Therefore, every covering sieve in the canonical topology is a universal colim sieve. 

To see that $\EuScript{U}$ is subcanonical, let $M$ be any object in $\EuScript{C}$ and consider the representable presheaf $rM$. 
For any $T\in \EuScript{U}(X)$,
\begin{equation*}\begin{split}
rM(X) & \cong rM\left(\colim_{T} U\right) \\
& \cong  \lim_{\substack{\longleftarrow \\ T}} (rM\circ U) \\
& \cong
\Eq\left(\prod_{A\xrightarrow{f} X\in T} F(A) \substack{\xrightarrow{\alpha} \\ \xrightarrow[\beta]{}} \prod_{\substack{B\xrightarrow{g} A \\ A\xrightarrow{f} X\in T}} F(B)\right)
\end{split}\end{equation*}
where the first isomorphism is because $T$ is a colim sieve, the second isomorphism is a general property of $\Hom_\EuScript{C}(-,M)$, and third isomorphim is fact (\ref{star}). Since this is true for every universal colim sieve $T$ and object $X$, then $rM$ is a sheaf. Therefore, all representable presheaves are sheaves in the universal colim sieve topology. 
\end{proof}
\bigskip


\noindent\textsc{Basis}
\bigskip

Now, for a very specific type of category, we give a basis for the canonical topology.

\begin{prop}\label{reducing sieve generating set}
Let $\EuScript{C}$ be a cocomplete category with pullbacks. 
Futher assume that coproducts and pullbacks commute in $\EuScript{C}$.
Then a sieve of the form $S = \langle\{f_\alpha\colon A_\alpha\to X\}_{\alpha\in\EuScript{A}}\rangle$ is a (universal) colim sieve if and only if the sieve $T = \langle\{\coprod f_\alpha\colon \coprod_{\alpha\in\EuScript{A}} A_\alpha\to X\}\rangle$ is a (universal) colim sieve.
\end{prop}

\begin{proof}
Fix $f\colon Y\to X$ and consider $f^\ast S$ and $f^\ast T$. Then
%
\begin{align*}
\colim_{f^\ast T}{U} & \cong \Coeq
\left(\begin{tikzcd}
\left(\left(\coprod_{\gamma\in\EuScript{A}} A_\gamma\right) \times_X Y\right)\times_Y \left(\left(\coprod_{\beta\in\EuScript{A}} A_\beta\right) \times_X Y\right)
\arrow[d, shift right = 2]
\arrow[d, shift left = 2] \\
\left(\coprod_{\alpha\in\EuScript{A}} A_\alpha\right) \times_X Y
\end{tikzcd}\right)\\
& \cong \Coeq
\left(\begin{tikzcd}
\left(\coprod_{\gamma\in\EuScript{A}} \left(A_\gamma \times_X Y\right) \right) \times_Y \left(\coprod_{\beta\in\EuScript{A}} \left(A_\beta \times_X Y\right)\right)
\arrow[d, shift right = 2]
\arrow[d, shift left = 2] \\
\coprod_{\alpha\in\EuScript{A}} \left(A_\alpha \times_X Y\right)
\end{tikzcd}\right)\\
& \cong \Coeq
\left(\begin{tikzcd}
\coprod_{\gamma,\beta\in\EuScript{A}}\left(\left( A_\gamma \times_X Y\right) \times_Y \left( A_\beta \times_X Y\right)\right)
\arrow[d, shift right = 2]
\arrow[d, shift left = 2] \\
\coprod_{\alpha\in\EuScript{A}} \left(A_\alpha \times_X Y\right)
\end{tikzcd}\right)\\
& \cong \colim_{f^\ast S}{U}
\end{align*}
by Lemma \ref{pb sieve gen set}, Proposition \ref{colim is coeq} and the commutativity of coproducts and pullbacks.
Therefore, $\colim_{f^\ast S}{U}\cong Y$ if and only if $\colim_{f^\ast T}{U}\cong Y$.
\end{proof}

\begin{thm}\label{hinting at basis for can top in special case}
Let $\EuScript{C}$ be a cocomplete category with pullbacks whose coproducts and pullbacks commute.
A sieve $S$ on $X$ is a (universal) colim sieve of $\EuScript{C}$ if and only if there exists some $\{A_\alpha\to X\}_{\alpha\in\EuScript{A}} \subset S$ where $\ds\coprod_{\alpha\in\EuScript{A}} A_\alpha \to X$ is a (universal) effective epimorphism.
\end{thm}

\begin{proof}
It is an easy application of Proposition \ref{reducing sieve generating set}, Corollary \ref{eff epi's gen colim sieves} and Theorem \ref{ucs is a top}.
\end{proof}

The above theorem shows us what our basis for the canonical topology should be; and indeed:

\begin{thm}\label{basis for can top in special case}
Let $\EuScript{C}$ be a cocomplete category with stable and disjoint coproducts and all pullbacks. 
For each $X$ in $\EuScript{C}$, define $K(X)$ by
$$\{A_\alpha\to X\}_{\alpha\in\EuScript{A}}\in K(X)\iff \coprod_{\alpha\in\EuScript{A}} A_\alpha \to X \text{ is a universal effective epimorphism.}$$
Then $K$ is a Grothendieck basis and generates the canonical topology on $\EuScript{C}$.
\end{thm}

\begin{proof}
We will use the universal colim sieve presentation (Theorem \ref{can top is ucs}).
For $K$ to be a basis we need three things:
\begin{enumerate}
\item $\{f\colon E\to X\}\in K(X)$ for every isomorphism $f$.
\item If $\{f_i\colon E_i\to X\}_{i\in I}\in K(X)$ and $g\colon Y\to X$, then $\{\pi_2\colon E_i\times_X Y\to Y\}_{i\in I}$ is in $K(Y)$
\item If $\{f_i\colon E_i\to X\}_{i\in I}\in K(X)$ and $\{g_{ij}\colon D_{ij}\to E_i\}_{j\in J_i}\in K(E_i)$ for each $i\in I$, then $\{f_i\circ g_{ij}\colon  D_{ij}\to X\}_{i\in I, j\in J_i}\in K(X)$.
\end{enumerate}

The first condition is true since isomorphisms are obviously universal effective epimorphisms.
The second condition follows from the fact that coproducts and pullbacks commute, and the assumed universal condition on $\coprod_{i\in I}E_i\to X$.
The third condition follows from Corollary \ref{eff epi composition} and Lemma \ref{coprod of eff epi's}.

Lastly, Theorem \ref{hinting at basis for can top in special case} showcases that this Grothendieck basis is indeed a basis for the canonical topology.
\end{proof}


\section{Universal Hocolim Sieves in the Category of Topological Spaces}\label{hcs in top}

In this section we explore some examples of universal hocolim sieves.
Let $\Delta$ be the cosimplicial indexing category;
in other words, the objects are the sets $[n] = \{0,\dots,n\}$ for $n>0$ and the morphisms are monotone increasing functions.
\bigskip

\noindent \textsc{Open Covers}
\bigskip

Let $X$ be a topological space with open cover $\EuScript{U}$.
Set
$$S(\EuScript{U}) \coloneqq \langle \{ V\subset X\ |\ V\in \EuScript{U}\} \rangle.$$ 
%
We will show that $S(\EuScript{U})$ is a universal hocolim sieve.

We start by recalling the \textit{\v{C}ech complex} $\check C (\EuScript{U})_\ast$ associated to the open cover $\EuScript{U}$.
This simplicial set is defined by $\check C (\EuScript{U})_n = \coprod V_{a_0}\cap\dots\cap V_{a_n}$ with the obvious face and degeneracy maps and $V_{a_i}\in \EuScript{U}$ for $i=0,\dots,n$.

Similarly, the \textit{\v{C}ech complex} of a set $B$ will be denoted by $\check C(B)_\ast$.
This simplicial set is defined by $\check C(B)_n = B^{n+1}$ with the obvious face and degeneracy maps. 
We remark that $\check C(B)_\ast$ is contracible (see \cite[][Proposition 3.12 and Example 3.14]{primer} and use $f\colon B\to\{\ast\}$). 

Additionally, for a simplicial set $K_\ast$ we define $\Delta(K_\ast)$ to be the Grothendieck construction for the functor $\gamma\colon \Delta\to\textbf{Sets}$ given by $[n]\mapsto K_n$.
In particular, $\Delta(K_\ast)$ is a category with objects
$([n],k)$ where $k\in K_n$.
We will abuse notation and write $k$ for the object $([n],k)$.

\begin{prop}\label{open covers are uhcs}
For any topological space $X$ and open cover $\EuScript{U}$, $S(\EuScript{U})$ is a universal hocolim sieve.
\end{prop}

\begin{proof}
Let $A$ be an indexing set for the cover $\EuScript{U}$, i.e.\ elements of $\EuScript{U}$ take the form $V_a$ for some $a\in A$.
Let $\Gamma\colon \Delta(\check C(A)_\ast)\to S(\EuScript{U})$ be defined by $\Gamma(a_0,\dots,a_n)$ equals $(V_{a_0}\cap\dots\cap V_{a_n}\xrightarrow{\iota} X)$ where $\iota$ is the inclusion map.

First we show that $\Gamma$ is a homotopy final functor (as defined by \cite{primer}).
Indeed, for a fixed $(f\colon Y\to X)\in S(\EuScript{U})$, 
$(f\downarrow\Gamma)$ is $\Delta(\check C(T)_\ast)$ where $T$ is the set $\coprod_{V\in\EuScript{U}}\left(\textbf{Top}\downarrow X\right)(Y,V)$ (using Notation \ref{note hom}) -- to see this, notice that any object in $(f\downarrow\Gamma)$ can be viewed (for some $n$) as an element of 
\begin{align*}
\coprod_{(a_0,\dots,a_n)}\left(\textbf{Top}\downarrow X\right)(Y,V_{a_0}\cap\dots\cap V_{a_n}) &\cong  \coprod_{(a_0,\dots,a_n)}\prod_{i=0}^n\left(\textbf{Top}\downarrow X\right)(Y,V_{a_i})\\
&\cong \prod_{i=0}^n \coprod_{V\in\EuScript{U}}\left(\textbf{Top}\downarrow X\right)(Y,V) \\
& = T^{n+1}.
\end{align*}
Since $(f\colon Y\to X)\in S(\EuScript{U})$, then $f$ factors through some $V \in \EuScript{U}$ and so $T$ is nonempty.
Therefore, the nerve of $\Delta(\check C(T)_\ast)$ is weakly equivalent to $\check C(T)_\ast$, which is itself contracible.

Since $\Gamma$ is homotopy final, then by \cite[][``Cofinality Theorem'']{primer},
\begin{equation}\label{open cover equation}
\hocolim{\Delta(\check C(A)_\ast)}{U\Gamma}\xrightarrow{\simeq} \hocolim{S(\EuScript{U})}{U}\rightarrow X.
\end{equation}
To see that the composition is a weak equivalence, we use the fact that $\Delta(\check C(A)_\ast)$ is a Grothendieck construction and therefore by \cite[][Theorem 26.8]{GrothendieckConstruction},
\begin{align*}
\hocolim{\Delta(\check C(A)_\ast)}{U\Gamma}
&\simeq \hocolim{[n]\in\Delta}{\hocolim{\check C(A)_n}{U\Gamma}}\\
&\simeq \hocolim{\Delta}{\check C(\EuScript{U})_\ast}
\end{align*}
where the last weak equivalence comes from the fact that $\check C(A)_n$ is a discrete category and hence 
$$\hocolim{\check C(A)_n}{U\Gamma}\xrightarrow{\simeq} \colim_{\check C(A)_n}{U\Gamma} = \coprod_{A^{n+1}}V_{a_0}\cap\dots\cap V_{a_n} = \check C(\EuScript{U})_n.$$
But by \cite[][Theorem 1.1]{dugger2001hypercovers}, $\hocolim{\ }{\check C} (\EuScript{U})_\ast \simeq X$.
Therefore, both the left map and the composition in (\ref{open cover equation}) are weak equivalences, which implies that the right map is too.

Universality follows immediately from Lemma \ref{pb sieve gen set} and the fact that the pullback on an open cover is an open cover.
\end{proof}

\bigskip

\noindent \textsc{Simplices Mapping into $X$}
\bigskip

For a topological space $X$, set
$$\Delta(X) \coloneqq \{ \Delta^n \to X\ |\ n\in \mathbb{Z}_{\geq 0}\},$$
%
i.e. all of the maps in $(\textbf{Top}\downarrow X)$ whose domain is a simplex. We will show that $\langle\Delta(X)\rangle$ is a universal hocolim sieve. First we recall a useful result from \cite[][Proposition 22.5]{primer}:

\begin{prop}\label{simplices hocolim}
For every topological space $X$, $\hocolim{\Delta(X)}{U}\to X$ is a weak equivalence.
\end{prop}



\begin{prop}\label{contain simplices imply hcs}
Any sieve $R$ on $X$ that contains $\Delta(X)$ is a hocolim sieve.
\end{prop}

\begin{proof}
Consider the inclusion functor $\alpha\colon \Delta(X)\to R$ and, for each $f\in R$, the natural morphism
$$\chi_f\colon\hocolim{(\alpha\downarrow f)}{U\mu_f}\to U(f)$$
%
where $\mu_f\colon (\alpha\downarrow f)\to R$ is the functor $(i,i\to f)\mapsto i$.

Notice that $(\alpha\downarrow f)$ and $\Delta(\text{domain}\ f)$ are equivalent categories.
Additionally, for all $(i,i\to f)\in (\alpha\downarrow f)$, $U\mu_f(i,i\to f) = \text{domain}\ i$.
Thus
$$\hocolim{(\alpha\downarrow f)}{U\mu_f} = \hocolim{\Delta(\text{domain}\ f)}{U}.$$
%
By Proposition \ref{simplices hocolim}, $\hocolim{\Delta(\text{domain}\ f)}{U} \to (\text{domain}\ f)$ is a weak equivalence.
Hence $\chi_f$ is a weak equivalence for all $f\in R$.

The above two paragraphs put us squarely in the hypotheses of \cite[][Theorem 6.9]{primer}, which means
we may now conclude that
$$\alpha_\#\colon\hocolim{\Delta(X)}{U\alpha}\to\hocolim{R}{U}$$
%
is a weak equivalence.
Moreover, up to abuse of notation, $U\alpha = U$, which by Proposition \ref{simplices hocolim} implies that $\hocolim{\Delta(X)}{U\alpha} \to X$ is a weak equivalence.
Thus in the composition
$$\hocolim{\Delta(X)}{U\alpha}\xrightarrow{\alpha_\#}\hocolim{R}{U}\rightarrow X$$
%
both the first arrow and the composition itself 
are weak equivalences.
Therefore $\hocolim{R}{U}\to X$ is also a weak equivalence.
\end{proof}

\begin{cor}\label{simplices are uhcs}
For any topological space $X$, $\langle\Delta(X)\rangle$ is a universal hocolim sieve.
\end{cor}

\begin{proof}
Let $f\colon Y\to X$ and consider $f^\ast\langle\Delta(X)\rangle$.
Clearly, $\Delta(Y)\subset f^\ast\langle\Delta(X)\rangle$.
Therefore, by Proposition \ref{contain simplices imply hcs}, $f^\ast\langle\Delta(X)\rangle$ is a hocolim sieve.
\end{proof}

Additionally, we remark that $\langle\Delta(X)\rangle$ is a colim sieve if and only if $X$ is a Delta-generated space.
Since not every space is Delta-generated, then for such an $X$, $\langle\Delta(X)\rangle$ is an example of a sieve in the homotopical canonical topology that is not in the canonical topology.

\begin{cor}
Let $\EuScript{U}$ be an open cover $X$. Let $R = \langle\{\Delta^n\to V\subset X\,|\, V\in\EuScript{U}\}\rangle$, i.e.\, $R$ is generated by the ``$\EuScript{U}$-small'' simplices. Then $R$ is a universal hocolim sieve.
\end{cor}

\begin{proof}
We will use the transitivity axiom from the definition of Grothendieck topology with $S(\EuScript{U})$, which
by Proposition \ref{open covers are uhcs} is in the homotopical canonical topology.
So 
we only need to show that $f^\ast R$ is a universal hocolim sieve for every $f\in S(\EuScript{U})$.

Fix $(f\colon Y\to X)\in S(\EuScript{U})$.
Then $f$ factors as $Y\xrightarrow{g} W\xrightarrow{i_W} X$ for some $W\in\EuScript{U}$ and inclusion map $i_W$.
Consider $i_W^\ast R = \langle\{\Delta^n\times_X W \to W\cap V \subset W\,|\, V\in\EuScript{U}\}\rangle$ (see Lemma \ref{pb sieve gen set}).
Notice that for any $(\Delta^n\to X)\in R$ that factors through $V\in\EuScript{U}$, $\Delta^n\times_X W \cong \Delta^n\times_V (W\cap V)$ --
now we apply the case $V=W$ to see that $\{\Delta^n\to W\}$ is part of $i_W^\ast R$'s generating set.
Therefore $\langle\Delta(W)\rangle\subset i_W^\ast R$.
But by Corollary \ref{simplices are uhcs}, $\langle\Delta(W)\rangle$ is in the homotopical canonical topology.
Since the homotopical canonical topology is a Grothendieck topology, then any sieve containing a cover is itself a cover.
Thus $i_W^\ast R$ is a universal hocolim sieve.
Hence $f^\ast R = g^\ast (i_W^\ast R)$ is a universal hocolim sieve.
\end{proof}

\bigskip

\noindent \textsc{Monogenic Sieves}
\bigskip

A sieve is called \textit{monogenic} if it can be generated by one morphism.
For $f\colon Y\to X$, let $\check C(f)_\ast$ be the \textit{\v{C}ech complex} on $f$.
In other words, $\check C(f)$ is the simplicial object of $\EuScript{M}$ defined by $\check C(f)_n = Y\times_X\dots \times_X Y$, i.e. the pullback of the $n$-tuple $(Y,\dots,Y)$ over $X$, with the obvious face and degeneracy maps.

\begin{prop}\label{reduction for monogenic}
For a simplicial model category $\EuScript{M}$,
let $S = \langle\{f\colon Y\to X\}\rangle$ be a sieve on $X$.
Then
$$\hocolim{S}{U} \simeq \hocolim{\ }{\check C}(f)_\ast.$$
\end{prop}

\begin{proof}[Sketch of Proof.]
This proof is similar to the proof of Proposition \ref{open covers are uhcs}.
Basically, $\Gamma\colon \Delta\to S$ defined by $[n]\mapsto (\check C(f)_n\to X)$ is homotopy final, which completes the proof.
Indeed, for any $(g\colon Z\to X)\in S$, $(g\downarrow\Gamma)$ is $\Delta(\check C(K)_\ast)$ where $K$ is the set $(\textbf{Top}\downarrow X)\left( Z, Y \right)$, which is both nonempty and contractible.
\end{proof}

\begin{prop}
If $f$ is locally split, then the sieve generated by $f$ is a universal hocolim sieve.
\end{prop}

\begin{proof}
Suppose $f$ is a locally split map, i.e. $f\colon Y\to X$ and there is an open cover $\EuScript{U}$ of $X$ such that for all $V\in\EuScript{U}$, $f\big|_{f^{-1}(V)}\colon f^{-1}(V) \to V$ is split.
Let $s_V\colon V\to f^{-1}(V)$ be the splitting map for $f\big|_{f^{-1}(V)}$.
Then the composition $V\xrightarrow{s_V} f^{-1}(V)\subset Y\xrightarrow{f} X$ equals the inclusion map $V\subset X$ and is in $\langle\{f\}\rangle$.
Indeed, $f\circ s_V = id_V$ and the composition clearly factors through $f$.
Thus $(V\subset X)\in \langle\{f\}\rangle$ for all $V\in \EuScript{U}$, which implies that $S(\EuScript{U})\subset \langle\{f\}\rangle$.
Since $S(\EuScript{U})$ is in the homotopical canonical topology (by Proposition \ref{open covers are uhcs}), then the Grothendieck topology transitivity axiom implies that any sieve containing it is also in the homotopical canonical topology.
Therefore, $\langle\{f\}\rangle$ is in the homotopical canonical topology.
\end{proof}

\bibliographystyle{plain}
\bibliography{CanonicalTopology}

\begin{thebibliography}{10}

\bibitem{GrothendieckConstruction}
Wojciech Chach{\'o}lski and J{\'e}r{\^o}me Scherer.
\newblock {\em Homotopy theory of diagrams}.
\newblock Number 736. American Mathematical Soc., 2002.

\bibitem{primer}
Daniel Dugger.
\newblock A primer on homotopy colimits.
\newblock {\em preprint}, 2008.

\bibitem{dugger2001hypercovers}
Daniel Dugger and Daniel~C Isaksen.
\newblock Hypercovers in topology.
\newblock {\em arXiv preprint math/0111287}, 2001.

\bibitem{goerss2009simplicial}
Paul~G Goerss and John~F Jardine.
\newblock {\em Simplicial homotopy theory}.
\newblock Springer Science \& Business Media, 2009.

\bibitem{johnstone2002sketches}
Peter~T Johnstone.
\newblock {\em Sketches of an elephant: A topos theory compendium}, volume~2.
\newblock Oxford University Press, 2002.

\bibitem{kashiwara2006categories}
Masaki Kashiwara and Pierre Schapira.
\newblock Categories and sheaves, volume 332 of.
\newblock {\em Grundlehren der Mathematischen Wissenschaften [Fundamental
  Principles of Mathematical Sciences]}, 2006.

\bibitem{kellymono}
GM~Kelly.
\newblock Monomorphisms, epimorphisms, and pull-backs.
\newblock {\em Journal of the Australian Mathematical Society},
  9(1-2):124--142, 1969.

\bibitem{cwm}
Saunders Mac~Lane.
\newblock {\em Categories for the working mathematician}, volume~5.
\newblock Springer Science \& Business Media, 2013.

\bibitem{maclane}
Saunders Mac~Lane and Ieke Moerdijk.
\newblock {\em Sheaves in geometry and logic: A first introduction to topos
  theory}.
\newblock Springer Science \& Business Media, 2012.

\bibitem{quillen1973higher}
Daniel Quillen.
\newblock Higher algebraic k-theory: I.
\newblock In {\em Higher K-theories}, pages 85--147. Springer, 1973.

\end{thebibliography}
\end{document}